\documentclass[12pt]{amsart}%
\usepackage{indentfirst, latexsym, bm, amsmath, eufrak, amsthm, amscdx,
amssymb, mathabx}
\usepackage[colorlinks, linkcolor=blue,citecolor=red]{hyperref}
\usepackage{upgreek}
\usepackage{bm}
\usepackage{url}
\usepackage{lineno}
\usepackage{amsmath}
\usepackage{amsfonts}
\usepackage{amssymb}
\usepackage{graphicx}%
\providecommand{\U}[1]{\protect\rule{.1in}{.1in}}
\theoremstyle{definition}
\theoremstyle{remark}
\numberwithin{equation}{section}
\newtheorem{theorem}{Theorem}[section]
\newtheorem*{theorem*}{Theorem}

\newtheorem{lemma}{Lemma}[section]
\newtheorem*{lemma*}{Lemma}
\newtheorem{corollary}{Corollary}[section]
\newtheorem{remark}{Remark}[section]
\newtheorem{proposition}{Proposition}[section]
\newtheorem{definition}{Definition}[section]

\ifx\pdfoutput\relax\let\pdfoutput=\undefined\fi
\newcount\msipdfoutput
\ifx\pdfoutput\undefined\else
\ifcase\pdfoutput\else
\ifx\paperwidth\undefined\else
\ifdim\paperheight=0pt\relax\else\pdfpageheight\paperheight\fi
\ifdim\paperwidth=0pt\relax\else\pdfpagewidth\paperwidth\fi
\fi\fi\fi
\begin{document}
\title[Existence of Conic Sasaki-Einstein Metrics ]{On the Existence of Conic Sasaki-Einstein Metrics on Log Fano Sasakian
Manifolds of Dimension Five}
\author{$^{\ast}$Shu-Cheng Chang}
\address{$^{\ast}$Mathematical Science Research Center, Chongqing University of
Technology, 400054, Chongqing, P.R. China}
\email{scchang@math.ntu.edu.tw}
\author{$^{\dag}$Fengjiang Li}
\address{$^{\dag}$Mathematical Science Research Center, Chongqing University of
Technology, 400054, Chongqing, P.R. China}
\email{lianyisky@163.com}
\author{$^{\dag\dag}$Chien Lin}
\address{$^{\dag\dag}$Mathematical Science Research Center, Chongqing University of
Technology, 400054, Chongqing, P.R. China}
\email{chienlin@cqut.edu.cn}
\author{$^{\ast\ast}$Chin-Tung Wu}
\address{$^{\ast\ast}$Department of Applied Mathematics, National Pingtung University,
Pingtung 90003, Taiwan}
\email{ctwu@mail.nptu.edu.tw }
\thanks{$^{\ast}$Shu-Cheng Chang is partially supported by Funds of MSRC, CQUT
0625199005. $^{\ast\ast}$Chin-Tung Wu is partially supported by NSTC
112-2115-M-153-001. $^{\dag}$Fengjiang Li is partially supported by NSFC
12301062, CQUT 2021ZDZ023 and CQEC KJZD-K202201102. $^{\dag\dag}$Chien Lin is
partially supported by KJQN202201165 and QN2022035003L}

\begin{abstract}
In this paper, we derive the uniform $L^{4}$-bound of the transverse conic
Ricci curvature along the conic Sasaki-Ricci flow on a compact transverse log
Fano Sasakian manifold $M$ of dimension five and the space of leaves of the
characteristic foliation is not well-formed. Then we first show that any
solution of the conic Sasaki-Ricci flow converges in the Cheeger-Gromov sense
to the unique singular orbifold conic Sasaki-Ricci soliton on
$M_{\mathrm{\infty}\text{ }}$which is a $\mathbb{S}^{1}$-orbibundle over the
unique singular conic K\"{a}hler-Ricci soliton on a log del Pezzo orbifold
surface. As a consequence, there exists a K\"{a}hler-Ricci soliton orbifold
metric on its leave space which is a log del Pezzo orbifold surface. Second,
we show that the conic Sasaki-Ricci soliton is the conic Sasaki-Einstein if
$M$ is transverse log $K$-polystable. In summary, we have the existence
theorems of orbifold Sasaki-Ricci solitons and Sasaki-Einstein metrics on a
compact quasi-regular Sasakian manifold of dimension five.

\end{abstract}
\maketitle
\tableofcontents

\section{Introduction}

Let $(M,\eta,\xi,\Phi,g)$ be a compact Sasakian manifold of dimension $2n+1.$
If the orbits of the Reeb vector field $\xi$ are all closed and hence circles,
then integrates to an isometric $U(1)$ action on $(M,g)$. Since it is nowhere
zero this action is locally free, that is, the isotropy group of every point
in $S$ is finite. If the $U(1$) action is in fact free then the Sasakian
structure is said to be regular. Otherwise, it is said to be quasi-regular. If
the orbits of $\xi$ are not all closed,\textbf{\ }the Sasakian structure is
said to be irregular. However, by the second structure theorem (\cite{ru},
\cite{bg}), any Sasakian structure $(\xi,\eta,\Phi)$ on $(M,g)$ is either
quasi-regular or there is a sequence of quasi-regular Sasakian structures
$(M,\xi_{i},\eta_{i},\Phi_{i},g_{i})$ converging in the compact-open
$C^{\infty}$-topology to $(\xi,\eta,\Phi,g).$ It means that there always
exists a quasi-regular Sasakian structure $(\xi,\eta,\Phi)$ on $(M,g)$.

A Sasakian $(2n+1)$-manifold is served as the odd-dimensional analogous of
K\"{a}hler manifolds. For instance, the K\"{a}hler cone of a Sasaki-Einstein
$5$-manifold is a Calabi-Yau $3$-fold. It provides interesting examples of the
AdS/CFT correspondence. On the other hand, the class of simply connected,
closed, oriented, smooth $5$-manifolds is classifiable under diffeomorphism
due to Smale-Barden \cite{s}, \cite{b}.

In a compact quasi-regular Sasakian manifold, the unit Reeb vector field
induces a\textbf{\ }$\mathbb{S}^{1}$-action which generates the finite
isotropy groups. It is the regular free action if the isotropy subgroup of
every point is trivial. In general, as in \cite{clw}, the space of leaves has
at least the codimension two fixed point set of every non-trivial isotropy
subgroup or the codimension one fixed point set of some non-trivial isotropy subgroup.

In a series of previous articles, it is our goal to address the related issues
on the geometrization and classification problems of Sasakian manifolds of
dimension five of Type I foliation singularities (\cite{chlw}, \cite{clw},
\cite{cchlw}) in the sense that the space of leaves has at least codimension
two fixed point set of every non-trivial isotropy subgroup.

In this article, we deal with the old dimensional counterpart of
Yau-Tian-Donaldson conjecture (\cite{cds1}, \cite{cds2}, \cite{cds3},
\cite{t5}) for conic Sasaki-Einstein metrics on a compact Sasakian manifolds
of dimension five of Type II foliation singularities in the sense that the
space of leaves admits the codimension one fixed point set of some non-trivial
isotropy subgroup. It can be viewed as a Sasaki analogue of Liu-Zhang
\cite{lz} for the conic K\"{a}hler-Ricci flow (\ref{2023}).

We recall the starting issue from our previous paper. We assume that $M$ is a
compact quasi-regular transverse Fano $5$-Sasakian manifold and the space $Z$
of leaves is a del Pezzo orbifold surface of cyclic quotient singularities
(well-formed) which means its orbifold singular locus and algebro-geometric
singular locus coincide, equivalently $Z$ has no branch divisors. Then we have

\begin{proposition}
\label{P11} (\cite{cchlw}) Let $(M,\xi,\eta_{0},g_{0})$ be a compact
quasi-regular transverse Fano Sasakian manifold of dimension five and
$(Z_{0}=M/\mathcal{F}_{\xi},h_{0},\omega_{h_{0}})$ denote the space of leaves
of the characteristic foliation which is a del Pezzo orbifold surface of
cyclic quotient klt singularities with codimension two orbifold singularities
$\Sigma_{0}$. Then, by a $D$-homothetic deformation, under the Sasaki-Ricci
flow
\[
\left\{
\begin{array}
[c]{lcl}%
\frac{\partial}{\partial t}\omega(t) & = & -\mathrm{Ric}^{\mathrm{T}}%
(\omega(t))+\omega(t),\\
\omega(0) & = & \omega_{0}.
\end{array}
\right.
\]
$(M(t),\xi,\eta(t),g(t))$ converges to a compact quasi-regular transverse Fano
Sasaki-Ricci soliton orbifold $(M_{\infty},\xi,\eta_{\infty},g_{\infty})$ with
the leave space of a del Pezzo orbifold surface $(Z_{\infty}=M_{\infty
}/\mathcal{F}_{\xi},h_{\infty})$ which can have at worst codimension two klt
orbifold singularities $\Sigma_{\infty}.$ Furthermore, $g^{T}(t_{i})$
converges to a Sasaki-Ricci soliton orbifold metric $g_{\infty}^{T}$ on
$M_{\infty}$ with $g_{\infty}^{T}=\pi^{\ast}(h_{\infty})$ such that
$h_{\infty}$ is the smooth gradient K\"{a}hler-Ricci soliton metric in the
Cheeger-Gromov topology on $Z_{\infty}\backslash\Sigma_{\infty}.$ In
additional, if $M$ is transverse $K$-stable, then $M(t)$ converges to a
compact transverse Fano Sasakian manifold $M_{\infty}$ which is isomorphic to
$M$ endowed with a smooth Sasaki-Einstein metric.
\end{proposition}

Now we focus on a compact quasi-regular Sasakian $5$-manifold of the Type II
foliation singularities. It means that its leave space $Z$ has the orbifold
structure $\mathcal{Z=}(Z,\Delta)$ with the codimension one fixed point set of
some non-trivial isotropy subgroup with branch divisors $\Delta$. Then, up to
conjugation, a representation $\mu_{p}$ of $\mathbb{Z}_{r}$ on $\mathbb{C}%
^{2}$ is conjugate to the $\mathbb{S}^{1}$-action
\[%
\begin{array}
[c]{c}%
\mu_{\mathbb{Z}_{r}}:(z_{1},z_{2})\rightarrow(e^{2\pi\sqrt{-1}\frac{a_{1}%
}{r_{1}}}z_{1},e^{2\pi\sqrt{-1}\frac{a_{2}}{r_{2}}}z_{2})
\end{array}
\]
for some positive integers $r_{1,}$ $r_{2}$ whose least common multiplier is
$r$, and $a_{i},$ $i=1,$ $2$ are integers coprime to $r_{i},$ $i=1,$ $2$. By
the first structure theory as in Proposition \ref{P21}, the foliation singular
set contains some $3$-dimensional Sasakian submanifolds of $M$, the
corresponding foliation singularities $\Delta^{T}$ in $(M,\Delta^{T}%
)$\textbf{\ }is the Hopf $\mathbb{S}^{1}$-orbibundle over a Riemann surface
$\Sigma_{h}.$ We refer to section $2$ in some detail (\cite{clw}).

More precisely, suppose that $(M,\xi,\eta,g)$ is a compact quasi-regular
Sasakian $5$-manifold. Then, by the first structure theorem, its leave space
of the characteristic foliation is a log del Pezzo orbifold surface
$(Z,\Delta)$ with klt singularities and
\[%
\begin{array}
[c]{c}%
\Delta=\sum_{\alpha=1}^{k}(1-\frac{1}{m_{\alpha}})\Delta_{\alpha},
\end{array}
\]
where $m_{\alpha}$ is called the ramification index of $\Delta_{\alpha}$.
Moreover, for its submersion
\[
\pi:(M,\Delta^{T},g,\omega)\rightarrow(Z,\Delta,h,\omega_{h})
\]
with the transverse K\"{a}hler structure $\frac{1}{2}d\eta=\pi^{\ast}%
(\omega_{h})=\omega$ and the transverse divisor (\cite{chlw}) $\Delta_{\alpha
}^{T}$ such that $\Delta^{T}=\pi_{\ast}^{-1}(\Delta)$ with
\[%
\begin{array}
[c]{c}%
\Delta^{T}=\sum_{\alpha=1}^{k}(1-\frac{1}{m_{\alpha}})\Delta_{\alpha}^{T}%
\end{array}
\]
and the canonical (determinant) bundle $K_{M}^{T}$ of $T^{1,0}(M)$ which is a
basic CR holomorphic line bundle (\cite{clw})
\[
K_{M}^{T}=\pi^{\ast}(K_{Z}^{orb})=\pi^{\ast}(K_{Z}+\Delta).
\]

However, it follows from Remark \ref{R21} that if there is a Sasaki-Einstein
metric $\omega$ on a log Fano quasi-regular\textbf{\ }Sasakian manifold
$(M,\Delta^{T})$ with klt foliation singularities, then $\widetilde{\phi
}^{\ast}\omega$ will be a so-called conic Sasaki-Einstein metric on a log Fano
regular Sasakian manifold $\widetilde{M}$\ with the transverse cone angle
$2\pi(1-\beta)$ along $D^{T}$ for a log resolution $\widetilde{\phi
}:(\widetilde{M},D^{T})\rightarrow(M,\Delta^{T}).$ Therefore we can turn
around to use this to construct a Sasaki-Einstein metric on a log Fano
quasi-regular Sasakian manifold $(M,\Delta^{T})$.

In view of the observation above, instead we shall work on the conic
Sasaki-Ricci flow
\begin{equation}
\left\{
\begin{array}
[c]{lcl}%
\frac{\partial}{\partial t}\omega(t) & = & -\mathrm{Ric}^{\mathrm{T}}%
(\omega(t))+\beta\omega(t)+(1-\beta)[D^{T}]\\
\omega(0) & = & \omega_{\ast}%
\end{array}
\right.  \label{2023}%
\end{equation}
in a compact regular transverse Fano Sasakian manifold $(M,\xi,\eta_{0}%
,\Phi_{0},g_{0},\omega_{0})$ of dimension five (subsection $2.1$ and $2.2$).
More precisely, it follows from Theorem \ref{T61}, Theorem \ref{T63}, and
Corollary \ref{C62} that we first derive the existence theorem of the gradient
conic Sasaki-Ricci soliton orbifold metric:

\begin{theorem}
\label{T66} Let $(M,\xi,\eta_{0},g_{0},(1-\beta)D^{T})$ be a compact
transverse log Fano regular Sasakian manifold of dimension five and
$(Z_{0}=M/\mathcal{F}_{\xi},h_{0},\omega_{h_{0}},(1-\beta)D)$ be the space of
leaves of the characteristic foliation which is a log del Pezzo surface with
the klt singularities. Then the solution $(M(t),\xi,\eta(t),g(t),(1-\beta
)D^{T})$ of the conic Sasaki-Ricci flow converges smoothly to a compact
transverse log Fano quasi-regular Sasakian orbifold manifold $(M_{\infty}%
,\xi,\eta_{\infty},g_{\infty},(1-\beta)D^{T})$ with the leave space of log del
Pezzo orbifold surface $(Z_{\infty}=M_{\infty}/\mathcal{F}_{\xi},h_{\infty
},(1-\beta)D)$ outside $\Sigma_{\infty}\cup(1-\beta)D$, where the
singularities $\Sigma_{\infty}$ is at worst of the codimension two orbifold
singularities $\Sigma_{\infty}$ in $Z_{\infty}.$ Furthermore, $g^{T}(t_{i})$
converges smoothly to a gradient conic Sasaki-Ricci soliton orbifold metric
$g_{\infty}^{T}$ on $M_{\infty}$ with $g_{\infty}^{T}=\pi^{\ast}(h_{\infty})$
such that $h_{\infty}$ is the smooth conic K\"{a}hler-Ricci soliton metric in
the Cheeger-Gromov topology on $Z_{\infty}\backslash(\Sigma_{\infty}%
\cup(1-\beta)D).$
\end{theorem}

All these results confirm the Hamilton-Tian conjecture that $(M,g(t))$
converges (along a subsequence) to a shrinking Sasaki-Ricci soliton with mild
singularities on a compact transverse Fano Sasakian manifold of dimension five.

Second, we show that it is indeed a conic Sasaki-Einstein metric if
$(M,(1-\beta)D^{T})$ is transverse log $K$-polystable (Definition \ref{d61}).
This is an old dimensional counterpart of Yau-Tian-Donaldson conjecture on a
compact log $K$-polystable K\"{a}hler manifold (\cite{cds1}, \cite{cds2},
\cite{cds3}, \cite{t5}). It can be viewed as a parabolic Sasaki analogue of
Tian-Wang \cite{tw1} and a generalized theorem of Chen-Sun-Wang \cite{csw},
Liu-Zhang \cite{lz} and Collins-Jacob \cite{cj}.

\begin{theorem}
\label{T12}Let $(M,\xi,\eta_{0},g_{0},(1-\beta)D^{T})$ be a compact transverse
log Fano regular Sasakian manifold of dimension five and $(Z_{0}%
=M/\mathcal{F}_{\xi},h_{0},\omega_{h_{0}},(1-\beta)D)$ be the space of leaves
of the characteristic foliation which is a log del Pezzo surface. If
$(M,(1-\beta)D^{T})$ is transverse log $K$-polystable, then there exists a
conic Sasaki-Einstein metric of $(M,(1-\beta)D^{T})$.
\end{theorem}

It follows from Remark \ref{R21} and Theorem \ref{T12} that we have the
existence theorem of Sasaki-Ricci soliton and Sasaki-Einstein metric on a
compact transverse log Fano quasi-regular Sasakian manifold $(M,\Delta^{T})$
of dimension five.

\begin{corollary}
\label{C12} Let $(M,\Delta^{T})$ be a compact quasi-regular transverse log
Fano Sasakian manifold of dimension five and $(Z,\Delta)$ be the space of
leaves of the characteristic foliation which is not well-formed and the
orbifold structure has the codimension one fixed point set of some non-trivial
isotropy subgroup. Then

\begin{enumerate}
\item There always exists a Sasaki-Ricci soliton orbifold metric on
$(M,\Delta^{T})$.

\item In additional, suppose that $(\widetilde{M},(1-\beta)\widetilde{D}^{T})$
is transverse log $K$-polystable for a log resolution $\widetilde{\phi
}:(\widetilde{M},D^{T})\rightarrow(M,\Delta^{T})$. Then there exists a
Sasaki-Einstein metric on $(M,\Delta^{T}).$
\end{enumerate}
\end{corollary}

Finally, it is proved by Wang-Zhu \cite{wz} that all smooth del Pezzo surfaces
admit a K\"{a}hler-Ricci soliton metric and Shi-Zhu \cite{sz} extended to tori
Fano orbifolds.\ However, as a consequence of Proposition \ref{P11}, Corollary
\ref{C12} and Guo-Phong-Song-Sturm \cite{gpss}, we have

\begin{corollary}
\label{C13}Let $(M,\Delta^{T})$ be a compact quasi-regular transverse log Fano
Sasakian manifold of dimension five and $(Z,\Delta)$ be the space of leaves of
the characteristic foliation which is a log del Pezzo orbifold surface. Then
there exists a K\"{a}hler-Ricci soliton orbifold metric on $(Z,\Delta)$.
\end{corollary}

It follow from Proposition \ref{P11} and Corollary \ref{C13} that

\begin{corollary}
There exists a Sasaki-Ricci soliton orbifold metric on any compact
quasi-regular transverse Fano or log Fano Sasakian manifold of dimension five.
\end{corollary}

In our upcoming project (\cite{cht}) for any compact irregular Fano Sasakian
manifold of dimension five, by the second structure theorem (\cite{ru},
\cite{bg}), we will focus on the compactness properties of a sequence of
orbifold Sasaki-Ricci solitons on a compact quasi-regular transverse Fano
Sasakian $5$-manifold and its leave space is normal projective Fano variety
with Kawamata log terminal (klt) singularities. More precisely, we are
expected to derive the Sasaki analogue of compactness theorems of compact
K\"{a}hler-Ricci solitons (\cite{gpss}, \cite{tz2}) by applying algebraic and
analytic invariants such as the $\alpha$-invariant due to Tian \cite{t3}, the
log canonical threshold due to Cheltsov-Shramov and Demailly \cite{chesh} and
a uniform positive lower bound of the log canonical threshold due to Birkar
\cite{bir}.

\begin{remark}
\label{R11}

\begin{enumerate}
\item It follows from Proposition \ref{P11} and Corollary \ref{C12} that we
have the existence theorem of Sasaki-Einstein metrics on a compact transverse
Fano quasi-regular Sasakian manifold of dimension five. Moreover, there are
non-regular Sasaki-Einstein metrics on connected sums of$\ \mathbb{S}%
^{2}\times\mathbb{S}^{3},$ rational homology $5$-spheres and connected sums of
these. We refer to \cite{bg}, \cite{k2}, \cite{k3}, \cite{sp}, \cite{cfo},
\cite{fow}, \cite{msy} and references therein.

\item In fact, most $5$-manifolds do not admit a positive Sasakian structure
and there is a classification of (quasi-regular) Sasaki-Einstein structures on
a rational homology $5$-sphere which are quasi-regular only. We refer to
\cite{bg}, \cite{k3}, \cite{su}, \cite{pw} and references therein.

\item The key ingredients for our proofs in this article remain hold for a
quasi-regular Sasakian manifold of dimension up to seven with its orbifold
leave space due to the $L^{4}$-Bound of the transverse conic Ricci curvature.

\item The same results hold for the conic K\"{a}hler-Ricci flow on normal log
pair of K\"{a}hler orbifolds of dimension up to three.

\item In general, there is a orbifold stability in a compact quasi-regular
Sasakian manifold in the sense of Ross-Thomas \cite{rt}. Instead, Collins and
Sz\'{e}kelyhidi \cite{csz2} showed that the Sasakian manifold admits a
Sasaki-Einstein metric if and only if its K\"{a}hler cone is $K$-stable by the
continuity method which corresponds to the Type I foliation singularities
(\cite{cchlw}). Along this direction, Boyer and van Coevering \cite{bvc} deal
with the Sasakian extremal metrics.
\end{enumerate}
\end{remark}

The crucial steps are followings. First, we are able to refine the Sasaki
analogue of a log pair quasi-regular Sasakian manifold $(M,\Delta^{T})$ with
klt foliation singularities. Then, for a Sasaki resolution $\widetilde{\phi
}:(\widetilde{M},D^{T})\rightarrow(M,\Delta^{T}),$ we can construct a
Sasaki-Einstein metric on such a log Fano quasi-regular Sasakian manifold $M$
with the foliation singularities $\Delta^{T}$ by finding the conic
Sasaki-Einstein metric on a log Fano regular Sasakian manifold $\widetilde{M}$
with the transverse cone angle $2\pi(1-\beta)$ along $D^{T}$. We refer to
Remark \ref{R21} in some detail.

Second, we study the structure and regularity theory of desired limit spaces
by applying $L^{p}$ version of the Cheeger-Colding-Tian structure theory
(\cite{cct}, \cite{ds}, \cite{tw1}, \cite{cht}), Then, based on Perelman's
uniform noncollapsing condition and pseudolocality theorem of the Ricci flow,
the limit solution must be a gradient conic Sasaki-Ricci soliton orbifold
metric $g_{\infty}^{T}$ on $M_{\infty}$. This is a Sasaki analogue of the
partial $C^{0}$-estimate (\cite{ds}, \cite[Theorem 5.1]{tz}) due to the first
structure theorem on quasi-regular Sasakian manifolds.

Finally, we are able to adapt the notions as in \cite{d2} and \cite{li} in our
Sasakian setting for the log Sasaki-Donaldson-Futaki invariant and log
Sasaki-Mabuchi $K$-energy. Then we can have the Sasaki analogue of log
$K$-polystable on Sasakian manifolds which is served as the different view
point from Collins-Sz\'{e}kelyhidi \cite{csz2} and Ross-Thomas \cite{rt}.

In section $2,$ we recall some preliminaries for Sasakian manifolds and
Sasakian structures. Then define a log pair $(M,\Delta^{T})$ with klt
foliation singularities. Based on this, we introduce the twisted and conic
Sasaki-Ricci flow. In section $3,$ we first prove the $L^{4}$-boundedness of
the transverse Ricci curvature along the twisted Sasaki-Ricci flows. Then, by
$L^{p}$ version of the Cheeger-Colding-Tian structure theory, to study the
structure of desired limit spaces. Finally, by a partial $C^{0}$-estimate, one
can refine the regularity of the desired limit space due to Donaldson
\cite{ds} and Tian-Wang \cite{tw1}. In the final section, we prove the main
theorems. More precisely, we derive all Sasaki analogues of log
Donaldson-Futaki invariant, log Mabuchi $K$-energy and Perelman's
$\mathcal{W}$-functional. We are able to obtain a shrinking Sasaki-Ricci
soliton with mild singularities and show that it is indeed a conic
Sasaki-Einstein metric if it is transverse log $K$-polystable.

\textbf{Acknowledgements.} Part of the project was done during the first named
author visiting to Department of Mathematics under the supported by iCAG
program of Higher Education Sprout Project of National Taiwan Normal
University and the Ministry of Education (MOE) in Taiwan. He would like to
express his gratitude for the warm hospitality there.

\section{Preliminaries}

In this section, we will recall some preliminaries for Sasakian manifolds and
Sasakian structures. Then define a log pair $(M,\Delta^{T})$ with klt
foliation singularities and introduce the twisted and conic Sasaki-Ricci flow.
We refer to \cite{bg}, \cite{fow}, \cite{sp}, \cite{clw}, and references
therein in some detail.

\subsection{Sasakian Geometry}

We review some basic facts from Sasakian geometry (\cite{bg}, \cite{fow},
\cite{chlw}). Let $(M,g)$ be a Riemannian $(2n+1)$-manifold. $(M,g)$ is called
Sasaki if the cone%
\[
(C(M),\overline{g},J):=(\mathbb{R}^{+}\times M\mathbf{,}dr^{2}+r^{2}g)
\]
is K\"{a}hler. In that case $\{r=1\}=\{1\}\times M\subset C(M)$. Define the
Reeb vector field
\[%
\begin{array}
[c]{c}%
\xi=J(\frac{\partial}{\partial r})
\end{array}
\]
and the contact $1$-form
\[
\eta(Y)=g(\xi,Y).
\]
Then $\eta(\xi)=1$ and $d\eta(\xi,X)=0.$ $\xi$ is a killing vector field with
unit length.

There are equivalent statements that there exists a Killing vector field $\xi$
of unit length on $M$ so that the tensor field of type $(1,1)$, defined by
\[
\Phi(Y)=\nabla_{Y}\xi,
\]
satisfies the condition%
\[
(\nabla_{X}\Phi)(Y)=g(\xi,Y)X-g(X,Y)\xi
\]
for any pair of vector fields $X$ and $Y$ on $M$. Furthermore, the Riemann
curvature of $(M,g)$ satisfies the condition%
\begin{equation}
R(X,\xi)Y=g(\xi,Y)X-g(X,Y)\xi. \label{a}%
\end{equation}
Now there is a natural splitting
\[%
\begin{array}
[c]{c}%
TC(M)=L_{r\frac{\partial}{\partial r}}\oplus L_{\xi}\oplus H
\end{array}
\]
and the relation%
\[%
\begin{array}
[c]{c}%
JY=\Phi(Y)-\eta(Y)r\frac{\partial}{\partial r}.
\end{array}
\]
Then%
\[
\Phi^{2}=-I+\eta\otimes\xi
\]
and%
\[
g(\Phi X,\Phi Y)=g(X,Y)-\eta(X)\eta(Y)\text{ \textrm{which is} }g=g^{T}%
+\eta\otimes\eta.
\]
Then $(M,\xi,\eta,g,\Phi)$ is said to be a Sasakian manifold $(M,g)$ with the
Sasakian structure $(\xi,\eta,g,\Phi)$. Let $\{U_{\alpha}\}_{\alpha\in\Lambda
}$ be an open covering of the Sasakian manifold $(M,\xi,\eta,g,\Phi)$ and
\[
\pi_{\alpha}:U_{\alpha}\rightarrow V_{\alpha}\subset%
\mathbb{C}
^{n}%
\]
submersion such that $\pi_{\alpha}\circ\pi_{\beta}^{-1}:\pi_{\beta}(U_{\alpha
}\cap U_{\beta})\rightarrow\pi_{\alpha}(U_{\alpha}\cap U_{\beta})$ is
biholomorphic. On each $V_{\alpha},$ there is a canonical isomorphism
\[
d\pi_{\alpha}:D_{p}\rightarrow T_{\pi_{\alpha}(p)}V_{\alpha}%
\]
for any $p\in U_{\alpha},$ where the contact subbundle $D=\ker\eta\subset TM.$
$d\eta$ is the Levi form on $D.$ It is a global two form on $M$.

Since $\xi$ generates isometries, the restriction of the Sasakian metric $g$
to $D$ gives a well-defined Hermitian metric $g_{\alpha}^{T}$ on $V_{\alpha}$.
In fact it is K\"{a}hler. The K\"{a}hler $2$-form $\omega_{\alpha}^{T}$ of the
Hermitian metric $g_{\alpha}^{T}$ on $V_{\alpha}$ is the same as the
restriction of the Levi form $\frac{1}{2}d\eta$ to $\widetilde{D_{\alpha}^{n}%
}$, the slice $\{x=$ \textrm{constant}$\}$ in $U_{\alpha}.$ The collection of
K\"{a}hler metrics $\{g_{\alpha}^{T}\}$ on $\{V_{\alpha}\}$ is so-called a
transverse K\"{a}hler metric. We often refer to $\frac{1}{2}d\eta$ as the
K\"{a}hler form of the transverse K\"{a}hler metric $g^{T} $ in the leaf space
$\widetilde{D^{n}}.$

We denote $\nabla^{T},$ $\mathrm{Rm}^{\mathrm{T}},$ $\mathrm{Ric}^{\mathrm{T}%
}$ and $\mathrm{R}^{\mathrm{T}}$ for its Levi-Civita connection, the
curvature, the Ricci tensor and the scalar curvature with respect to $g^{T}$.
For $\widetilde{X},$ $\widetilde{Y},$ $\widetilde{W},$ $\widetilde{Z}\in
\Gamma(TD)$ and the $d\pi_{\alpha}$-corresponding $X,$ $Y,$ $W,$ $Z\in
\Gamma(TV_{\alpha}),$ we have the following:%
\[%
\begin{array}
[c]{c}%
\nabla_{X}^{T}Y=d\pi_{\alpha}(\nabla_{\widetilde{X}}\widetilde{Y}),\text{
}\widetilde{\nabla_{X}^{T}Y}=\nabla_{\widetilde{X}}\widetilde{Y}+g(JX,Y)\xi,
\end{array}
\]%
\begin{equation}%
\begin{array}
[c]{ccl}%
\mathrm{Rm}^{\mathrm{T}}(X,Y,Z,W) & = & \mathrm{Rm}(\widetilde{X}%
,\widetilde{Y},\widetilde{Z},\widetilde{W})+2g(J\widetilde{X},\widetilde
{Y})g(J\widetilde{Z},\widetilde{W})\\
&  & +g(J\widetilde{X},\widetilde{Z})g(J\widetilde{Y},\widetilde
{W})-g(J\widetilde{X},\widetilde{W})g(J\widetilde{Y},\widetilde{Z})
\end{array}
\label{b}%
\end{equation}
and
\begin{equation}
\mathrm{Ric}^{\mathrm{T}}(X,Z)=\mathrm{Ric}(\widetilde{X},\widetilde
{Z})+2g(\widetilde{X},\widetilde{Z}). \label{c}%
\end{equation}
Furthermore, the Sasaki metric can be written as
\[
g=g^{T}\oplus\eta\otimes\eta
\]
and transversal Ricci curvature $\mathrm{Ric}^{\mathrm{T}}$ is
\[
\mathrm{Ric}^{\mathrm{T}}=\mathrm{Ric}+2g^{T}.
\]

In terms of the foliation normal coordinate which are simultaneously foliated
and Riemann normal coordinates (\cite{gkn}, \cite{chlw}) along the leaves, we
have a coordinate $\{x,z^{1},z^{2},\cdots,z^{n}\}$ on a neighborhood $U$ of
$p$, such that $x$ is the coordinate along the leaves with $\xi=\frac
{\partial}{\partial x}$ on $U$. Moreover, $\{z^{1},z^{2},\cdots,z^{n}\}$ is
the local holomorphic coordinates on $V_{\alpha}$. We pull back these to
$U_{\alpha}$ and still write the same. Then we have the foliation local
coordinate $\{x,z^{1},z^{2},\cdots,z^{n}\}$ on $U_{\alpha}\ $and $(D\otimes%
\mathbb{C}
)$ is spanned by the fields $Z_{j}=\frac{\partial}{\partial z^{j}}+\sqrt
{-1}h_{j}\frac{\partial}{\partial x},\ j\in\left\{  1,2,\cdots,n\right\}  $
with
\[%
\begin{array}
[c]{c}%
\eta=dx-\sqrt{-1}h_{j}dz^{j}+\sqrt{-1}h_{\overline{j}}d\overline{z}^{j}%
\end{array}
\]
and its dual frame
\[
\{\eta,dz^{j},\ j=1,2,\cdots,n\}.
\]
Here $h$ is a basic function such that $\frac{\partial h}{\partial x}=0$ and
$h_{j}=\frac{\partial h}{\partial z^{j}},$ $h_{j\overline{l}}=\frac
{\partial^{2}h}{\partial z^{j}\partial\overline{z}^{l}}$ with the foliation
normal coordinate%
\begin{equation}
h_{j}(p)=0,\text{ }h_{j\overline{l}}(p)=\delta_{j}^{l},\text{ }dh_{j\overline
{l}}(p)=0. \label{AAA3}%
\end{equation}

Now in terms of the normal coordinate, we have%
\[%
\begin{array}
[c]{c}%
g^{T}=g_{i\overline{j}}^{T}dz^{i}d\overline{z}^{j}%
\end{array}
\]
and
\[%
\begin{array}
[c]{c}%
\omega=2\sqrt{-1}h_{i\overline{j}}dz^{i}\wedge d\overline{z}^{j}.
\end{array}
\]
Here $g_{i\overline{j}}^{T}=g^{T}(\frac{\partial}{\partial z^{i}}%
,\frac{\partial}{\partial\overline{z}^{j}})=d\eta(\frac{\partial}{\partial
z^{i}},\Phi\frac{\partial}{\partial\overline{z}^{j}})=2h_{i\overline{j}}.$
Furthermore,
\begin{equation}%
\begin{array}
[c]{c}%
R_{i\overline{j}}^{T}=-\frac{\partial^{2}}{\partial z^{i}\partial\overline
{z}^{j}}\log\det(g_{\alpha\overline{\beta}}^{T})
\end{array}
\label{A}%
\end{equation}
and the transversal Ricci form $\rho^{T}$
\[%
\begin{array}
[c]{c}%
\rho^{T}=\mathrm{Ric}^{\mathrm{T}}(J\cdot,\cdot)=-\sqrt{-1}R_{i\overline{j}%
}^{T}dz^{i}\wedge d\overline{z}^{j}.
\end{array}
\]
Finally, we define the basic form $\gamma$ in Sasakian manifold so that%
\[
i(\xi)\gamma=0\text{\ \textrm{and} }\mathcal{L}_{\xi}\gamma=0.
\]
In this case the closure of the $1$-parameter subgroup of the isometry group
of $(M,g)$ is isomorphic to a torus $\mathbb{T}^{k}$, for some positive
integer $k$ called the rank of the Sasakian structure. In particular,
irregular Sasakian manifolds have at least a $\mathbb{T}^{2}$ isometry.

The first structure theory of Sasakian manifolds reads as

\begin{proposition}
\label{P21}(\cite{ru}, \cite{sp}, \cite{bg}) Let $(M,\eta,\xi,\Phi,g)$ be a
compact quasi-regular Sasakian manifold of dimension $2n+1$ and $Z$ denote the
space of leaves of the characteristic foliation $\mathcal{F}_{\xi}$ (just as
topological space). Then

\begin{enumerate}
\item $Z$ carries the structure of a Hodge orbifold $\mathcal{Z=}(Z,\Delta)$
with an orbifold K\"{a}hler metric $h$ and K\"{a}hler form $\omega$ which
defines an integral class in $H_{orb}^{2}(Z,\mathbb{Z}\mathbf{)}$ in such a
way that $\pi:(M,g,\omega)\rightarrow(Z,h,\omega_{h})$ is an orbifold
Riemannian submersion, and a principal $\mathbb{S}^{1}$-orbibundle
($V$-bundle) over $Z.$ Furthermore,it satisfies $\frac{1}{2}d\eta=\pi^{\ast
}(\omega_{h}).$ The fibers of $\pi$ are geodesics.

\item $Z$ is also a $\mathbb{Q}$-factorial, polarized, normal projective
algebraic variety.

\item The orbifold $Z$ is Fano if and only if $\mathrm{Ric}_{g}>-2$. In this
case $Z$ as a topological space is simply connected; and as an algebraic
variety is uniruled with Kodaira dimension $-\infty$.

\item $(M,\xi,g)$ is Sasaki-Einstein if and only if $(Z,h)$ is
K\"{a}hler-Einstein with scalar curvature $4n(n+1).$

\item If $(M,\eta,\xi,\Phi,g)$ is regular then the orbifold structure is
trivial and $\pi$ is a principal circle bundle over a smooth projective
algebraic variety.

\item As real cohomology classes, there is a relation between the first basic
Chern class and the first orbifold Chern class
\[
c_{1}^{B}(M):=c_{1}(\emph{F}_{\xi})=\pi^{\ast}c_{1}^{orb}(Z).
\]

\item Conversely, let $\pi:M\rightarrow Z$ be a $\mathbb{S}^{1}$-orbibundle
over a compact Hodge orbifold $(Z,h)$ whose first Chern class is an integral
class defined by $[\omega_{Z}]$, and $\eta$ be a $1$-form with $\frac{1}%
{2}d\eta=\pi^{\ast}\omega_{Z}$. Then $(M,\pi^{\ast}h+\eta\otimes\eta)$ is a
Sasakian manifold if all the local uniformizing groups inject into the
structure group $U(1).$
\end{enumerate}
\end{proposition}

\subsection{Log Pair with the KLT Foliation Singularities}

In this subsection, we follow the notations as in \cite{k1}, \cite{tw1},
\cite{clw} and references therein. Let $(M,\xi,\alpha_{\xi})$ be a compact
quasi-regular Sasakian $5$-manifold and the leave space $Z:=M/\mathbb{S}^{1}$
be an orbifold K\"{a}hler surface. Denote by the quotient map
\[
\pi:M\rightarrow Z
\]
as before, and call such a $5$-manifold $(M,\xi)$ an $\mathbb{S}^{1}%
$-orbibundle. More precisely, $M$ admits a locally free, effective
$\mathbb{S}^{1}$-action
\[%
\begin{array}
[c]{c}%
\alpha_{\xi}:\mathbb{S}^{1}\times M\rightarrow M
\end{array}
\]
such that $\alpha_{\xi}(t)$ is orientation-preserving, for every
$t\in\mathbb{S}^{1}$. Since $M$ is compact, $\alpha_{\xi}$ is proper and the
isotropy group $\Gamma_{p}$ of every point $p\in M$ is finite.

\begin{definition}
The principal orbit type $M_{\operatorname{reg}}$ corresponds to points in
$(M,\xi,\alpha_{\xi})$ with the trivial isotropy group and
$M_{\operatorname{reg}}\rightarrow M_{\operatorname{reg}}/\mathbb{S}^{1}$ is a
principle $\mathbb{S}^{1}$-bundle. Furthermore, the orbit $\mathbb{S}_{p}^{1}$
of a point $p\in M$ is called a regular fiber if $p\in M_{\operatorname{reg}}%
$, and a singular fiber otherwise. In this case, $M_{\mathrm{sing}}%
/\mathbb{S}^{1}\simeq\Sigma^{orb}(Z).$
\end{definition}

\begin{proposition}
(\cite{clw}) Let $(M,\eta,\xi,\Phi,g)$ be a compact quasi-regular Sasakian
$5$-manifold and $Z$ be its leave space of the characteristic foliation. Then
$Z$ is a $\mathbb{Q}$-factorial normal projective algebraic orbifold surface satisfying

\begin{enumerate}
\item if its leave space $(Z,\emptyset)$ has at least codimension two fixed
point set of every non-trivial isotropy subgroup. That is to say $Z$ is
well-formed, then $Z$ has isolated singularities of a finite cyclic quotient
of $\mathbb{C}^{2}$ and the action is
\[%
\begin{array}
[c]{c}%
\mu_{\mathbb{Z}_{r}}:(z_{1},z_{2})\rightarrow(\zeta^{a}z_{1},\zeta^{b}z_{2}),
\end{array}
\]
where $\zeta$ is a primitive $r$-th root of unity. We denote the cyclic
quotient singularity by $\frac{1}{r}(a,b)$ with $(a,r)=1=(b,r)$. In
particular, the action can be rescaled so that every cyclic quotient
singularity corresponds to a $\frac{1}{r}(1,a)$-type singularity with
$(r,a)=1,$ $\zeta=e^{2\pi\sqrt{-1}\frac{1}{r}}$. In particular, it is klt
(Kawamata log terminal) singularities. Moreover, the corresponding
singularities in $(M,\eta,\xi,\Phi,g)$\ is called\ foliation cyclic quotient
singularities of type\textbf{\ }%
\[%
\begin{array}
[c]{c}%
\frac{1}{r}(1,a)
\end{array}
\]
at a singular fibre $\mathbb{S}_{p}^{1}$ in $M$.

\item if its leave space $(Z,\Delta)$ has the codimension one fixed point set
of some non-trivial isotropy subgroup. Then the action is
\[%
\begin{array}
[c]{c}%
\mu_{\mathbb{Z}_{r}}:(z_{1},z_{2})\rightarrow(e^{2\pi\sqrt{-1}\frac{a_{1}%
}{r_{1}}}z_{1},e^{2\pi\sqrt{-1}\frac{a_{2}}{r_{2}}}z_{2}),
\end{array}
\]
for some positive integers $r_{1,}$ $r_{2}$ whose least common multiplier is
$r$, and $a_{i},$ $i=1,$ $2$ are integers coprime to $r_{i},$ $i=1,$ $2$. Then
the foliation singular set contains some $3$-dimensional submanifolds of $M.$
More precisely, the corresponding singularities in $(M,\eta,\xi,\Phi
,g)$\textbf{\ }is called\textbf{\ }the Hopf $\mathbb{S}^{1}$-orbibundle over a
Riemann surface $\Sigma_{h}.$
\end{enumerate}
\end{proposition}

Next we recall that a log pair $(Z,\Delta)$ is consisting of a connected
compact projective normal variety $Z$ and an effective $\mathbb{Q}$-divisor
$\Delta$ such that $(K_{Z}+\Delta)$ is $\mathbb{Q}$-Cartier. Along this
direction, we define a log pair $(M,\Delta^{T})$ with klt foliation singularities.

\begin{definition}
Let $(Z,\Delta)$ be a log pair with klt singularities and a log minimal
resolution $\phi:(Y,D_{Y})\rightarrow(Z,\Delta)$.

\begin{enumerate}
\item It means that there exists a unique $\mathbb{Q}$-divisor $D_{Y}=\sum
_{j}b_{j}D_{j}$ with $b_{j}>-1$ such that
\[
K_{Y}=\phi^{\ast}(K_{Z}+\Delta)+D_{Y},\text{ }D_{Y}=E+\Delta_{Y}^{\prime},
\]
where the union of the $\phi$-exceptional divisor $E$ and the strict transform
$\Delta_{Y}^{\prime}$ of $\Delta$ is a simple normal crossing divisor defined
by
\[%
\begin{array}
[c]{c}%
E:=\sum_{0<b_{j}}b_{j}D_{j}=\phi^{-1}(Z_{\text{\textrm{sing}}})\text{\ }%
\end{array}
\]
and
\[%
\begin{array}
[c]{c}%
\Delta_{Y}^{\prime}:=\phi_{\ast}^{-1}\Delta=\sum_{-1<b_{j}\leq0}b_{j}D_{j}.
\end{array}
\]

\item A conic K\"{a}hler metric $\omega$ on a compact K\"{a}hler $n$-manifold
$Z$ along the irreducible divisor $D$ with cone angle $2\pi\beta,$ $0<\beta<1$
is asymptotically equivalent to the model metric on $\mathbb{C}^{n}$
\[%
\begin{array}
[c]{c}%
\omega_{\beta}=\sqrt{-1}\frac{dz_{1}\wedge d\overline{z}_{1}}{|z_{1}%
|^{2(1-\beta)}}+\sum_{j=2}^{n}dz_{j}\wedge d\overline{z}_{j}.
\end{array}
\]
It is called conic K\"{a}hler-Einstein
\[
\mathrm{Ric}(\omega):=\beta\omega(t)+(1-\beta)[D]
\]
in the sense of currents with $D\thicksim-|K_{M}|$. That is, locally near a
point $p\in Z,$ there is a holomorphic coordinate system $(V,z_{1}%
,\cdots,z_{n})$ such that $D\cap V=\{z\in V;$ $z_{1}=0\}.$
\end{enumerate}
\end{definition}

Now given a log minimal resolution $\phi:(Y,D_{Y})\rightarrow(Z,\Delta)$ as
above. Since $Y$ is nonsingular, by the first Sasakian structure theorem
again, there exists a submersion $\widetilde{\pi}:(\widetilde{M}%
,D^{T})\rightarrow(Y,D_{Y})$ with the regular Sasakian manifold $(\widetilde
{M},D^{T}).$ Then we define the so-called log pair with the klt foliation
singularities as follows:

\begin{definition}
Let $(M,\xi,\eta,g)$ be a compact quasi-regular Sasakian manifold with the
submersion $\pi:(M,\Delta^{T})\rightarrow(Z,\Delta)$. Given $\phi,$ $\pi$ and
$\widetilde{\pi}$ as above, the transverse log resolution is defined by
$\widetilde{\phi}:(\widetilde{M},D^{T})\rightarrow(M,\Delta^{T})$ so that it
is basic and the following diagram
\[%
\begin{array}
[c]{ccc}%
(\widetilde{M},\widetilde{E}^{T},(\Delta_{Y}^{\prime})^{T},D_{\widetilde{M}%
}^{T}) & \overset{\widetilde{\phi}}{\longrightarrow} & (M,\Delta^{T})\\
\downarrow\widetilde{\pi} & \circlearrowright & \downarrow\pi\\
(Y,E,\Delta_{Y}^{\prime},D_{Y}) & \overset{\phi}{\longrightarrow} & (Z,\Delta)
\end{array}
\]
is commutative $\pi\circ\widetilde{\phi}=\phi\circ\widetilde{\pi}.$

\begin{enumerate}
\item We call $(M,\Delta^{T})$ is a log pair with klt foliation singularities
if there exists a unique transverse $\mathbb{Q}$-divisor
\[%
\begin{array}
[c]{c}%
D_{\widetilde{M}}^{T}=\sum_{j}b_{j}D_{j}^{T}%
\end{array}
\]
with $b_{j}>-1$ such that
\[%
\begin{array}
[c]{c}%
K_{\widetilde{M}}^{T}=\widetilde{\phi}^{\ast}(K_{M}^{T})+D_{\widetilde{M}}%
^{T},\text{ }D_{\widetilde{M}}^{T}=\widetilde{E}^{T}+(\Delta_{Y}^{\prime}%
)^{T}.
\end{array}
\]
That is
\[%
\begin{array}
[c]{c}%
K_{\widetilde{M}}^{T}=\widetilde{\phi}^{\ast}(K_{M}^{T})+\sum_{-1<b_{j}\leq
0}b_{j}D_{j}^{T}+\widetilde{E}^{T}.
\end{array}
\]
Here the $\widetilde{\phi}$-exceptional divisor $\widetilde{E}^{T}$
\[%
\begin{array}
[c]{c}%
\widetilde{\phi}^{-1}(\pi^{-1}(Z_{\text{\textrm{sing}}}))=\widetilde{E}%
^{T}=\sum_{0<b_{j}}b_{j}D_{j}^{T}%
\end{array}
\]
with $\pi(M_{\text{\textrm{sing}}})=M_{\text{\textrm{sing}}}/\mathbb{S}%
^{1}=\Sigma^{orb}(Z)$ and the strict transform $(\Delta_{Y}^{\prime})^{T}$ of
the transverse branch divisor $\Delta^{T}$
\[%
\begin{array}
[c]{c}%
\widetilde{\phi}_{\ast}^{-1}(\Delta^{T})=\Delta^{\prime T}=\sum_{-1<b_{j}%
\leq0}b_{j}D_{j}^{T}.
\end{array}
\]
Set $a_{j}:=-b_{j},$ if $-1<b_{j}\leq0.$ Then
\[%
\begin{array}
[c]{c}%
K_{Y}^{-1}=\phi^{\ast}(K_{Z}^{-1}-\Delta)+\sum_{0\leq a_{j}<1}a_{j}D_{j}-E
\end{array}
\]
and%
\[%
\begin{array}
[c]{c}%
(K_{\widetilde{M}}^{T})^{-1}-\sum_{0\leq a_{j}<1}a_{j}D_{j}^{T}=\widetilde
{\phi}^{\ast}((K_{M}^{T})^{-1}))-(\Delta_{Y}^{\prime})^{T}-\widetilde{E}%
^{T}=\widetilde{\phi}^{\ast}((K_{M}^{T})^{-1}-\Delta^{T}).
\end{array}
\]

\item $(M,\Delta^{T})$ is a log Fano quasi-regular Sasakian manifold if
$(K_{M}^{T})^{-1}-\Delta^{T}$ is a basic ample $\mathbb{Q}$-line bundle. That
is
\[
c_{1}^{B}(M,\Delta^{T})=c_{1}^{B}(M)-c_{1}^{B}(\Delta^{T})>0.
\]

\item $(\widetilde{M},D^{T})$ is a log Fano regular Sasakian manifold if
$D^{T}$ is a normal crossing smooth divisors such that $L^{T}$ is a basic
ample $\mathbb{Q}$-line bundle:
\[%
\begin{array}
[c]{c}%
L^{T}:=(K_{\widetilde{M}}^{T})^{-1}-\sum_{0\leq a_{j}<1}a_{j}D_{j}%
^{T}=\widetilde{\phi}^{\ast}((K_{M}^{T})^{-1}-\Delta^{T}).
\end{array}
\]

\item Let $(\widetilde{M},D^{T})$ be a log Fano regular Sasakian manifold. A
conic Sasakian metric $\widetilde{\omega}$ is called a conic Sasaki-Einstein
on $(\widetilde{M},D^{T})$ with the transverse cone angle $2\pi(1-a_{j})$
along $D_{j}^{T}$ with $\widetilde{\pi}_{\ast}^{-1}(D_{Y})=D_{j}^{T}$ if
\[%
\begin{array}
[c]{c}%
\widetilde{\mathrm{Ric}}^{\mathrm{T}}(\widetilde{\omega})=\widetilde{\omega
}+\sum_{0<a_{j}<1}a_{j}D_{j}^{T}%
\end{array}
\]
such that $2\pi(1-a_{j})$ is the cone angle along $D_{Y}$ with $\widetilde
{\pi}_{\ast}^{-1}(D_{Y})=D_{j}^{T}.$
\end{enumerate}
\end{definition}

\begin{remark}
\label{R21} If there is a Sasaki-Einstein metric $\omega$ on a log Fano
quasi-regular Sasakian manifold $(M,\Delta^{T})$ with klt foliation
singularities, then $\widetilde{\phi}^{\ast}\omega$ will be a degenerate conic
Sasaki-Einstein metric on a log Fano regular Sasakian manifold $\widetilde{M}$
with the transverse cone angle $2\pi(1-a_{j})$ along $D_{j}^{T}$:
\[%
\begin{array}
[c]{c}%
\mathrm{Ric}^{\mathrm{T}}(\widetilde{\phi}^{\ast}\omega)=\widetilde{\phi
}^{\ast}\omega+\sum_{0<a_{j}<1}a_{j}D_{j}^{T}.
\end{array}
\]
Then, for a log Fano $(M,\Delta^{T})$ with such a resolution $\widetilde{\phi
}:(\widetilde{M},D^{T})\rightarrow(M,\Delta^{T})$, one might expect that we
can turn around to use this to construct a Sasaki-Einstein metric on a log
Fano quasi-regular Sasakian manifold $M$ with the foliation singularities
$\Delta^{T}$. Indeed, this is true due to Chang-Han-Tie \cite{cht} and
Tian-Wang \cite{tw2} on the work of the compactness and existence of conic
Sasaki-Einstein metrics.
\end{remark}

\subsection{The Twisted and Conic Sasaki-Ricci Flow}

Let $(M,D^{T})$ is a log Fano regular Sasakian manifold with a smooth basic
divisors $D^{T}$. We consider the transverse K\"{a}hler class which is not
proportional to the basic first Chern class such that
\[
c_{1}^{B}(M)-\beta\lbrack\omega_{0}]_{B}=[\alpha]_{B}\neq0.
\]
Then for any fixed basic closed $(1,1)$-form $\theta\in\lbrack\alpha]_{B}$, it
is natural to ask if there exists a unique transverse K\"{a}hler form
$\omega\in\lbrack\omega_{0}]_{B}$ so that
\[
\mathrm{Ric}^{\mathrm{T}}(\omega)=\beta\omega+\theta
\]
which is called a twisted Sasaki-Ricci metric. The corresponding twisted
Sasaki-Ricci flow reads as%
\begin{equation}%
\begin{array}
[c]{l}%
\frac{\partial}{\partial t}\omega(t)=-\mathrm{Ric}^{\mathrm{T}}(\omega
(t))+\beta\omega(t)+\theta
\end{array}
\label{2023-2}%
\end{equation}
for $0<\beta<1$.

Let $\left(  M,\eta,g_{0},\left(  1-\beta\right)  D^{T}\right)  $ be a
Sasakian manifold with the transverse cone angle $2\pi\beta$ along $D^{T}$ and
$D^{T}\thicksim-K_{M}^{T}$ such that
\[%
\begin{array}
[c]{c}%
c_{1}^{B}(M,\left(  1-\beta\right)  D^{T})=c_{1}^{B}(M)-\left(  1-\beta
\right)  c_{1}^{B}([D^{T}])=\beta\lbrack\omega_{0}]_{B}>0
\end{array}
\]
where $\left[  D^{T}\right]  $ is the line bundle induced by the transverse
divisor $D^{T}$ and $\omega_{0}=d\eta$ denotes the transverse K\"{a}hler form
of $g_{0}.$ Now we consider the so-called conic Sasaki-Ricci flow%
\[
\left\{
\begin{array}
[c]{lcl}%
\frac{\partial}{\partial t}\omega(t) & = & -\mathrm{Ric}^{\mathrm{T}}%
(\omega(t))+\beta\omega(t)+(1-\beta)[D^{T}],\\
\omega(0) & = & \omega_{\ast}.
\end{array}
\right.
\]
Here $\omega_{\ast}=\omega_{0}+\delta\sqrt{-1}\partial_{B}\overline{\partial
}_{B}||S^{T}||^{2\beta}$ and $S^{T}$ is the basic defined section of $D^{T}$
in transverse Fano log pair $(M,D^{T}),$ $[D^{T}]$ is the current of
integration along $D^{T}$.

\section{Twisted Sasaki-Ricci Flow on Transverse Log Fano Sasakian Manifolds}

\subsection{$L^{4}$-Bound of the Transverse Conic Ricci Curvature}

In this subsection, we consider the $L^{4}$-boundedness of the transverse
Ricci curvature along the twisted Sasaki-Ricci flow (\ref{2023-1}) as follows.

Let $\left(  M,\eta,g_{0},\left(  1-\beta\right)  D^{T}\right)  $ be a
Sasakian manifold with the transverse cone angle $2\pi\beta$ along $D^{T}.$ In
this article, we use the following twisted Sasaki-Ricci flow to approach the
conic Sasaki-Ricci flow (\ref{2023})
\begin{equation}%
\begin{array}
[c]{c}%
\left\{
\begin{array}
[c]{ccl}%
\frac{\partial\omega^{\epsilon}}{\partial t} & = & -\left(  \rho
_{\omega^{\epsilon}}^{T}-\theta_{\epsilon}\right)  +\beta\omega^{\epsilon}\\
\omega^{\epsilon}\left(  0\right)  & = & \widehat{\omega}_{0}^{\epsilon}%
\equiv\omega_{0}+\sqrt{-1}k\partial_{B}\overline{\partial}_{B}\chi_{\epsilon
}(\epsilon^{2}+\left\Vert S\right\Vert _{h}^{2}).
\end{array}
\right.
\end{array}
\label{2023-1}%
\end{equation}
Here $k$ is a sufficiently small positive number, $D^{T}=\left[  S=0\right]
$, $h$ is a Hermitian metric on $\left[  D^{T}\right]  $,
\[%
\begin{array}
[c]{c}%
\chi_{\epsilon}\left(  \epsilon^{2}+t\right)  =\frac{1}{\beta}\int_{0}%
^{t}\frac{\left(  \epsilon^{2}+r\right)  ^{\beta}-\epsilon^{2\beta}}{r}dr
\end{array}
\]
and
\begin{equation}%
\begin{array}
[c]{c}%
\theta_{\epsilon}=\left(  1-\beta\right)  [\omega_{0}+\sqrt{-1}\partial
_{B}\overline{\partial}_{B}\log(\left\Vert S\right\Vert _{h}^{2}+\epsilon
^{2})].
\end{array}
\label{2023-4}%
\end{equation}
Then the parabolic Monge-Ampere equations for potentials are
\begin{equation}%
\begin{array}
[c]{c}%
\left\{
\begin{array}
[c]{lcl}%
\frac{\partial\varphi_{\varepsilon}}{\partial t} & = & \log\frac
{\omega_{\varphi_{\varepsilon}}^{n}\wedge\eta_{0}}{\omega_{\varepsilon}%
^{n}\wedge\eta_{0}}+F_{\varepsilon}+\beta(k\chi+\varphi_{\varepsilon}),\\
\varphi_{\varepsilon}(0) & = & c_{\varepsilon}(0).
\end{array}
\right.
\end{array}
\label{2023-1A}%
\end{equation}
Here $c_{\varepsilon}(0)$ are uniformly bounded for $\varepsilon$
\[%
\begin{array}
[c]{c}%
c_{\varepsilon}(0)=\frac{1}{\beta}\int_{0}^{\infty}e^{-\beta s}||\nabla
^{T}u_{\varepsilon}(s)||_{L^{2}}^{2}ds-\frac{1}{V}\int_{M}F_{\varepsilon}%
d\mu_{\varepsilon}-\frac{k\beta}{V}\int_{M}\chi d\mu_{\varepsilon}%
\end{array}
\]
and
\[%
\begin{array}
[c]{c}%
F_{\varepsilon}=F_{0}+\log[\frac{\omega_{_{\varepsilon}}^{n}\wedge\eta_{0}%
}{\omega_{0}^{n}\wedge\eta_{0}}(|S^{T}|^{2}+\varepsilon^{2})^{1-\beta}]
\end{array}
\]
with $F_{0}$ is the Ricci potential of $\omega_{0}$ so that $-\mathrm{Ric}%
^{\mathrm{T}}(\omega_{0})+\omega_{0}=\sqrt{-1}\partial_{B}\overline{\partial
}_{B}F_{0}$ with $\int e^{-F_{0}}dV_{0}=V.$

It follows from the standard argument that we have the long-time existence of
the solution to the twisted Sasaki-Ricci flow (\ref{2023-1}). Note that
$\omega^{\epsilon}\left(  t\right)  $ evolves in the same K\"{a}hler class
$\left[  \omega_{0}\right]  $. So $\omega^{\epsilon}\left(  t\right)  $ could
be expressed in terms of the transverse K\"{a}hler potential%
\[%
\begin{array}
[c]{c}%
\omega^{\epsilon}\left(  t\right)  =\widehat{\omega}_{0}^{\epsilon}+\sqrt
{-1}\partial_{B}\overline{\partial}_{B}\varphi_{\epsilon}\left(  t\right)  .
\end{array}
\]
Take $u_{\epsilon}\left(  t\right)  $ to be the transverse twisted Ricci
potential, i.e.
\begin{equation}%
\begin{array}
[c]{c}%
-\left(  \rho_{\omega^{\epsilon}}^{T}-\theta_{\epsilon}\right)  +\beta
\omega^{\epsilon}=\sqrt{-1}\partial_{B}\overline{\partial}_{B}u_{\epsilon
}\left(  t\right)
\end{array}
\label{U1}%
\end{equation}
which is normalized by
\[%
\begin{array}
[c]{c}%
\frac{1}{V}\int_{M}\exp\left(  -u_{\epsilon}\left(  t\right)  \right)
d\mu_{\epsilon}\left(  t\right)  =1
\end{array}
\]
where $V=\int_{M}d\mu_{\epsilon}\left(  t\right)  $ and $d\mu_{\epsilon
}\left(  t\right)  =\left(  \omega^{\epsilon}\left(  t\right)  \right)
^{n}\wedge\eta$. It's easy to see $\int_{M}d\mu_{\epsilon}\left(  t\right)  $
is independent of the choice of $\epsilon$ and $t$.

\begin{proposition}
\label{UP2}We have the equality%
\[%
\begin{array}
[c]{c}%
\frac{1}{2}\square_{B}[\left(  \Delta_{B}u_{\epsilon}\left(  t\right)
\right)  ^{2}]=-\left\vert \nabla^{T}\Delta_{B}u_{\epsilon}\left(  t\right)
\right\vert _{\omega^{\epsilon}\left(  t\right)  }^{2}+\Delta_{B}u_{\epsilon
}\left(  t\right)  \langle\rho_{\omega^{\epsilon}}^{T}-\theta_{\epsilon}%
,\sqrt{-1}\partial_{B}\overline{\partial}_{B}u_{\epsilon}\left(  t\right)
\rangle_{\omega^{\epsilon}\left(  t\right)  }.
\end{array}
\]
Here $\Delta_{B}$ means the basic Laplacian with respect to $\omega^{\epsilon
}\left(  t\right)  $,$\ \square_{B}$ denotes the operator $(\frac{\partial
}{\partial t}-\Delta_{B})$.
\end{proposition}

\begin{proof}
Due to $\frac{\partial\omega^{\epsilon}}{\partial t}=\sqrt{-1}\partial
_{B}\overline{\partial}_{B}u_{\epsilon}\left(  t\right)  $, there is a
constant $c\left(  t\right)  $ which depends only on the parameter $t$ such
that $\overset{\centerdot}{\varphi}_{\epsilon}\left(  t\right)  =u_{\epsilon
}\left(  t\right)  +c\left(  t\right)  $. Here after the quantity with the dot
$\cdot$ means its time-derivative. From
\[%
\begin{array}
[c]{ccl}%
\sqrt{-1}\partial_{B}\overline{\partial}_{B}\overset{\centerdot}{u}_{\epsilon}
& = & -\overset{\centerdot}{\rho}_{\omega^{\epsilon}}^{T}+\beta\overset
{\centerdot}{\omega}^{\epsilon}\\
& = & \sqrt{-1}\partial_{B}\overline{\partial}_{B}u_{\epsilon}\left(
t\right)  +\sqrt{-1}\partial_{B}\overline{\partial}_{B}\Delta_{B}u_{\epsilon
}\left(  t\right)  ,
\end{array}
\]
it could be deduced that $\square_{B}u_{\epsilon}=\beta u_{\epsilon
}-a_{\varepsilon}\left(  t\right)  $ for
\begin{equation}%
\begin{array}
[c]{c}%
a_{\varepsilon}\left(  t\right)  =\frac{\beta}{V}\int_{M}u_{\epsilon}\left(
t\right)  \exp\left(  -u_{\epsilon}\left(  t\right)  \right)  d\mu_{\epsilon
}\left(  t\right)  .
\end{array}
\label{2023-6}%
\end{equation}
Because
\[%
\begin{array}
[c]{c}%
\frac{\partial}{\partial t}\left(  \Delta_{B}u_{\epsilon}\right)
=\langle\left(  \rho_{\omega^{\epsilon}}^{T}-\theta_{\epsilon}\right)
-\beta\omega^{\epsilon},\sqrt{-1}\partial_{B}\overline{\partial}%
_{B}u_{\epsilon}\left(  t\right)  \rangle_{\omega^{\epsilon}\left(  t\right)
}+\Delta_{B}\left(  \Delta_{B}u_{\epsilon}+\beta u_{\epsilon}-a_{\varepsilon
}\left(  t\right)  \right)  ,
\end{array}
\]
we obtain
\[%
\begin{array}
[c]{c}%
\square_{B}\left(  \Delta_{B}u_{\epsilon}\right)  =\langle\left(  \rho
_{\omega^{\epsilon}}^{T}-\theta_{\epsilon}\right)  ,\sqrt{-1}\partial
_{B}\overline{\partial}_{B}u_{\epsilon}\left(  t\right)  \rangle
_{\omega^{\epsilon}\left(  t\right)  }.
\end{array}
\]
This implies that
\[%
\begin{array}
[c]{c}%
\frac{1}{2}\square_{B}[\left(  \Delta_{B}u_{\epsilon}\left(  t\right)
\right)  ^{2}]=-\left\vert \nabla^{T}\Delta_{B}u_{\epsilon}\left(  t\right)
\right\vert _{\omega^{\epsilon}\left(  t\right)  }^{2}+\Delta_{B}u_{\epsilon
}\left(  t\right)  \langle\rho_{\omega^{\epsilon}}^{T}-\theta_{\epsilon}%
,\sqrt{-1}\partial_{B}\overline{\partial}_{B}u_{\epsilon}\left(  t\right)
\rangle_{\omega^{\epsilon}\left(  t\right)  }.
\end{array}
\]

\end{proof}

\begin{proposition}
\label{UP3}The functions $\nabla^{T}\nabla^{T}u_{\epsilon}$, $\nabla
^{T}\overline{\nabla}^{T}u_{\epsilon}$, and $|\mathrm{Rm}_{\omega^{\epsilon
}\left(  t\right)  }^{\mathrm{T}}|_{\omega^{\epsilon}}$ all lie in
$L^{2}\left(  M,\omega^{\epsilon}\right)  $.
\end{proposition}

\begin{proof}
It suffices to observe the inequality%
\[%
\begin{array}
[c]{c}%
\int_{M}\left\vert \theta_{\epsilon}\right\vert _{\omega^{\epsilon}\left(
t\right)  }^{2}d\mu_{\epsilon}\left(  t\right)  \leq\int_{M}tr_{\omega
^{\epsilon}\left(  t\right)  }\theta_{\epsilon}d\mu_{\epsilon}\left(
t\right)  =\int_{M}n\theta_{\epsilon}\wedge d\mu_{\epsilon}\left(  t\right)
\rightarrow n\int_{D^{T}}d\mu_{0}\left(  t\right)
\end{array}
\]
as $\epsilon\longrightarrow0^{+}$; hence, we have
\[%
\begin{array}
[c]{c}%
\int_{M}\left\vert \theta_{\epsilon}\right\vert _{\omega^{\epsilon}\left(
t\right)  }^{2}d\mu_{\epsilon}\left(  t\right)  \leq C.
\end{array}
\]

\end{proof}

\begin{remark}
Actually, the inequality of the first line in the proof as above is the equality.
\end{remark}

\begin{lemma}
\label{L32}For any $t\geq1$, there exists a positive constant $C$ which is
independent of $\epsilon$ and $t$ such that
\[%
\begin{array}
[c]{c}%
\int_{M\times\left[  t,t+1\right]  }\left\vert \nabla^{T}\Delta_{B}%
u_{\epsilon}\left(  t\right)  \right\vert _{\omega^{\epsilon}\left(  t\right)
}^{2}d\mu_{\epsilon}\left(  t\right)  \leq C.
\end{array}
\]

\end{lemma}

\begin{proof}
By Proposition \ref{UP2} and the integration by parts, we see that
\[%
\begin{array}
[c]{cl}
& \int_{M}\left\vert \nabla^{T}\Delta_{B}u_{\epsilon}\left(  t\right)
\right\vert _{\omega^{\epsilon}\left(  t\right)  }^{2}d\mu_{\epsilon}\left(
t\right) \\
\leq & \int_{M}|\langle\rho_{\omega^{\epsilon}}^{T}-\theta_{\epsilon}%
,\sqrt{-1}\partial_{B}\overline{\partial}_{B}u_{\epsilon}\left(  t\right)
\rangle_{\omega^{\epsilon}\left(  t\right)  }|\left\vert \left(  \Delta
_{B}u_{\epsilon}\left(  t\right)  \right)  \right\vert d\mu_{\epsilon}\left(
t\right) \\
& -\frac{1}{2}\int_{M}\frac{\partial}{\partial t}\left(  \Delta_{B}%
u_{\epsilon}\left(  t\right)  \right)  ^{2}d\mu_{\epsilon}\left(  t\right) \\
\leq & \int_{M}|\langle\rho_{\omega^{\epsilon}}^{T}-\theta_{\epsilon}%
,\sqrt{-1}\partial_{B}\overline{\partial}_{B}u_{\epsilon}\left(  t\right)
\rangle_{\omega^{\epsilon}\left(  t\right)  }|\left\vert \left(  \Delta
_{B}u_{\epsilon}\left(  t\right)  \right)  \right\vert d\mu_{\epsilon}\left(
t\right) \\
& -\frac{1}{2}\frac{d}{dt}\int_{M}\left(  \Delta_{B}u_{\epsilon}\left(
t\right)  \right)  ^{2}d\mu_{\epsilon}\left(  t\right)  +\frac{1}{2}\int
_{M}\left(  \Delta_{B}u_{\epsilon}\left(  t\right)  \right)  ^{2}%
|\mathrm{R}_{\omega^{\epsilon}\left(  t\right)  }^{\mathrm{T}}-tr_{\omega
^{\epsilon}\left(  t\right)  }\theta_{\epsilon}-\beta n|d\mu_{\epsilon}\left(
t\right)  .
\end{array}
\]
Then, integrating both sides over the interval $\left[  t,t+1\right]  $, the
estimate%
\[%
\begin{array}
[c]{c}%
\int_{M\times\left[  t,t+1\right]  }\left\vert \nabla^{T}\Delta_{B}%
u_{\epsilon}\left(  t\right)  \right\vert _{\omega^{\epsilon}\left(  t\right)
}^{2}d\mu_{\epsilon}\left(  t\right)  \leq C
\end{array}
\]
holds with the help of Proposition \ref{UP1} and Proposition \ref{UP3}.
\end{proof}

As in \cite{lz}, the straightforward calculations conclude the following inequalities:

\begin{proposition}
\label{UP4}There exists a universal positive constant $C=C(\omega_{0})$ such
that for $t\geq1$ and $\epsilon>0$

\begin{enumerate}
\item
\[%
\begin{array}
[c]{l}%
\int_{M}|\nabla^{T}\overline{\nabla}^{T}u_{\epsilon}|_{\omega^{\epsilon
}\left(  t\right)  }^{4}d\mu_{\epsilon}\left(  t\right) \\
\leq C\int_{M}[|\overline{\nabla}^{T}\nabla^{T}\nabla^{T}u_{\epsilon}%
|_{\omega^{\epsilon}\left(  t\right)  }^{2}+|\nabla^{T}\nabla^{T}%
\overline{\nabla}^{T}u_{\epsilon}|_{\omega^{\epsilon}\left(  t\right)  }%
^{2}]d\mu_{\epsilon}\left(  t\right)  .
\end{array}
\]

\item
\[%
\begin{array}
[c]{l}%
\int_{M}|\nabla^{T}\nabla^{T}u_{\epsilon}|_{\omega^{\epsilon}\left(  t\right)
}^{4}d\mu_{\epsilon}\left(  t\right) \\
\leq C\int_{M}[|\overline{\nabla}^{T}\nabla^{T}\nabla^{T}u_{\epsilon}%
|_{\omega^{\epsilon}\left(  t\right)  }^{2}+|\nabla^{T}\nabla^{T}%
\overline{\nabla}^{T}u_{\epsilon}|_{\omega^{\epsilon}\left(  t\right)  }%
^{2}+|\nabla^{T}\nabla^{T}\nabla^{T}u_{\epsilon}|_{\omega^{\epsilon}\left(
t\right)  }^{2}]d\mu_{\epsilon}\left(  t\right) \\
\leq C\int_{M}[|\nabla^{T}\Delta_{B}u_{\epsilon}|_{\omega^{\epsilon}\left(
t\right)  }^{2}+|\nabla^{T}\nabla^{T}u_{\epsilon}|_{\omega^{\epsilon}\left(
t\right)  }^{2}+|\mathrm{Rm}_{\omega^{\epsilon}\left(  t\right)  }%
^{\mathrm{T}}|_{\omega^{\epsilon}\left(  t\right)  }^{2}]d\mu_{\epsilon
}\left(  t\right)  .
\end{array}
\]

\end{enumerate}
\end{proposition}

\begin{theorem}
\label{T31} There exists a positive constant $C$ which is independent of
$\epsilon$ and $t$ such that%
\[%
\begin{array}
[c]{c}%
\int_{M}\left\vert \rho_{\omega^{\epsilon}}^{T}-\theta_{\epsilon}\right\vert
_{\omega^{\epsilon}\left(  t\right)  }^{4}d\mu_{\epsilon}\left(  t\right)
\leq C
\end{array}
\]

for any $\epsilon>0$ and $t\geq1$.
\end{theorem}

\begin{proof}
By the equation (\ref{U1}), we only need to justify the uniform boundedness of
$\int_{M}|\sqrt{-1}\partial_{B}\overline{\partial}_{B}u_{\epsilon}\left(
t\right)  |_{\omega^{\epsilon}\left(  t\right)  }^{4}d\mu_{\epsilon}\left(
t\right)  $. From the equality
\[%
\begin{array}
[c]{c}%
\square_{B}\left(  \Delta_{B}u_{\epsilon}\right)  =\langle\rho_{\omega
^{\epsilon}}^{T}-\theta_{\epsilon},\sqrt{-1}\partial_{B}\overline{\partial
}_{B}u_{\epsilon}\left(  t\right)  \rangle_{\omega^{\epsilon}\left(  t\right)
}=\beta\Delta_{B}u_{\epsilon}-|\nabla^{T}\overline{\nabla}^{T}u_{\epsilon
}|_{\omega^{\epsilon}\left(  t\right)  }^{2},
\end{array}
\]

we have
\[%
\begin{array}
[c]{cl}
& \frac{\partial}{\partial t}\left\vert \nabla^{T}\Delta_{B}u_{\epsilon
}\left(  t\right)  \right\vert _{\omega^{\epsilon}\left(  t\right)  }^{2}\\
= & -u_{\epsilon}^{j\overline{k}}\left(  t\right)  \nabla_{j}^{T}\Delta
_{B}u_{\epsilon}\nabla_{\overline{k}}^{T}\Delta_{B}u_{\epsilon}+\nabla_{j}%
^{T}\Delta_{B}\overset{\centerdot}{u}_{\epsilon}\nabla_{\overline{j}}%
^{T}\Delta_{B}u_{\epsilon}+\nabla_{j}^{T}\Delta_{B}u_{\epsilon}\nabla
_{\overline{j}}^{T}\Delta_{B}\overset{\centerdot}{u}_{\epsilon}\\
= & -u_{\epsilon}^{j\overline{k}}\left(  t\right)  \nabla_{j}^{T}\Delta
_{B}u_{\epsilon}\nabla_{\overline{k}}^{T}\Delta_{B}u_{\epsilon}+\nabla_{j}%
^{T}(\Delta_{B}^{2}u_{\epsilon}+\beta\Delta_{B}u_{\epsilon})\nabla
_{\overline{j}}^{T}\Delta_{B}u_{\epsilon}\\
& +\nabla_{j}^{T}\Delta_{B}u_{\epsilon}\nabla_{\overline{j}}^{T}(\Delta
_{B}^{2}u_{\epsilon}+\beta\Delta_{B}u_{\epsilon}-|\nabla^{T}\overline{\nabla
}^{T}u_{\epsilon}|_{\omega^{\epsilon}\left(  t\right)  }^{2}-\nabla_{j}%
^{T}|\nabla^{T}\overline{\nabla}^{T}u_{\epsilon}|_{\omega^{\epsilon}\left(
t\right)  }^{2})\\
= & -u_{\epsilon}^{j\overline{k}}\left(  t\right)  \nabla_{j}^{T}\Delta
_{B}u_{\epsilon}\nabla_{\overline{k}}^{T}\Delta_{B}u_{\epsilon}-(\nabla
_{j}^{T}|\nabla^{T}\overline{\nabla}^{T}u_{\epsilon}|_{\omega^{\epsilon
}\left(  t\right)  }^{2})\nabla_{\overline{j}}^{T}\Delta_{B}u_{\epsilon}\\
& -\nabla_{j}^{T}\Delta_{B}u_{\epsilon}(\nabla_{\overline{j}}^{T}|\nabla
^{T}\overline{\nabla}^{T}u_{\epsilon}|_{\omega^{\epsilon}\left(  t\right)
}^{2})+w_{j}w_{\overline{j}i\overline{i}}\\
& +w_{\overline{j}}\{w_{ji\overline{i}}-[\left(  \theta_{\epsilon}\right)
_{j\overline{k}}+\beta\left(  g_{\epsilon}\right)  _{j\overline{k}}-\left(
u_{\epsilon}\right)  _{j\overline{k}}]w_{k}\}+2\beta|\nabla_{j}^{T}\Delta
_{B}u_{\epsilon}|_{\omega^{\epsilon}\left(  t\right)  }^{2}\\
= & -(\nabla_{j}^{T}|\nabla^{T}\overline{\nabla}^{T}u_{\epsilon}%
|_{\omega^{\epsilon}\left(  t\right)  }^{2})\nabla_{\overline{j}}^{T}%
\Delta_{B}u_{\epsilon}-\nabla_{j}^{T}\Delta_{B}u_{\epsilon}(\nabla
_{\overline{j}}^{T}|\nabla^{T}\overline{\nabla}^{T}u_{\epsilon}|_{\omega
^{\epsilon}\left(  t\right)  }^{2})\\
& +w_{\overline{j}}w_{ji\overline{i}}+w_{j}w_{\overline{j}i\overline{i}%
}-\left(  \theta_{\epsilon}\right)  _{j\overline{k}}w_{\overline{j}}%
w_{k}+\beta\left\vert \nabla^{T}\Delta_{B}u_{\epsilon}\right\vert
_{\omega^{\epsilon}\left(  t\right)  }^{2}\\
= & -(\nabla_{j}^{T}|\nabla^{T}\overline{\nabla}^{T}u_{\epsilon}%
|_{\omega^{\epsilon}\left(  t\right)  }^{2})\nabla_{\overline{j}}^{T}%
\Delta_{B}u_{\epsilon}-\nabla_{j}^{T}\Delta_{B}u_{\epsilon}(\nabla
_{\overline{j}}^{T}|\nabla^{T}\overline{\nabla}^{T}u_{\epsilon}|_{\omega
^{\epsilon}\left(  t\right)  }^{2})\\
& +\Delta_{B}\left\vert \nabla^{T}\Delta_{B}u_{\epsilon}\left(  t\right)
\right\vert _{\omega^{\epsilon}\left(  t\right)  }^{2}-|\nabla^{T}\nabla
^{T}\Delta_{B}u_{\epsilon}|_{\omega^{\epsilon}\left(  t\right)  }^{2}%
-|\nabla^{T}\overline{\nabla}^{T}\Delta_{B}u_{\epsilon}|_{\omega^{\epsilon
}\left(  t\right)  }^{2}\\
& -\left(  \theta_{\epsilon}\right)  _{j\overline{k}}w_{\overline{j}}%
w_{k}+\beta\left\vert \nabla^{T}\Delta_{B}u_{\epsilon}\right\vert
_{\omega^{\epsilon}\left(  t\right)  }^{2}%
\end{array}
\]
for $w=\Delta_{B}u_{\epsilon}$; then, by considering the quantity $\frac
{d}{dt}\int_{M}\left\vert \nabla^{T}\Delta_{B}u_{\epsilon}\left(  t\right)
\right\vert _{\omega^{\epsilon}\left(  t\right)  }^{2}d\mu_{\epsilon}\left(
t\right)  $, it could be found that
\begin{equation}%
\begin{array}
[c]{cl}
& \frac{d}{dt}\int_{M}\left\vert \nabla^{T}\Delta_{B}u_{\epsilon}\left(
t\right)  \right\vert _{\omega^{\epsilon}\left(  t\right)  }^{2}d\mu
_{\epsilon}\left(  t\right) \\
\leq & -\frac{1}{2}\int_{M}[|\nabla^{T}\overline{\nabla}^{T}\Delta
_{B}u_{\epsilon}|_{\omega^{\epsilon}\left(  t\right)  }^{2}+|\nabla^{T}%
\nabla^{T}\Delta_{B}u_{\epsilon}|_{\omega^{\epsilon}\left(  t\right)  }%
^{2}]d\mu_{\epsilon}\left(  t\right) \\
& -\int_{M}\left(  \theta_{\epsilon}\right)  _{j\overline{k}}w_{\overline{j}%
}w_{k}d\mu_{\epsilon}\left(  t\right)  +C\int_{M}\left\vert \nabla^{T}%
\Delta_{B}u_{\epsilon}\right\vert _{\omega^{\epsilon}\left(  t\right)  }%
^{2}d\mu_{\epsilon}\left(  t\right)  +C\\
\leq & C\int_{M}\left\vert \nabla^{T}\Delta_{B}u_{\epsilon}\right\vert
_{\omega^{\epsilon}\left(  t\right)  }^{2}d\mu_{\epsilon}\left(  t\right)  +C
\end{array}
\label{U2}%
\end{equation}
from the fact that $\sqrt{-1}\partial_{B}\overline{\partial}_{B}%
\log(\left\Vert S\right\Vert _{h}^{2}+\epsilon^{2})\geq0$ and Proposition
\ref{UP4}.

Finally, for any $\epsilon>0$ and $t\geq1$, the uniform boundedness of
\[%
\begin{array}
[c]{c}%
\int_{M}\left\vert \rho_{\omega^{\epsilon}}^{T}-\theta_{\epsilon}\right\vert
_{\omega^{\epsilon}\left(  t\right)  }^{4}d\mu_{\epsilon}\left(  t\right)
\leq C
\end{array}
\]
follows from (\ref{U2}), Lemma \ref{L32} and Proposition \ref{UP4}.
\end{proof}

\begin{remark}
It is worth to note that it is may not true for the boundedness of
$\theta_{\epsilon}$ in general; however, we are able to control its integral bound!
\end{remark}

\subsection{Longtime Behavior under the Twisted Sasaki-Ricci Flow}

All transverse quantities on Sasakian manifolds such as the transverse
Perelman's $\mathcal{W}$-functional can be viewed as their K\"{a}hler
counterparts restricted on basic forms and transverse K\"{a}hler structure.
Hence, under the twisted or conic Sasaki-Ricci flow, the Reeb vector field and
the transverse holomorphic structure are both invariant, and the metrics are
bundle-like. One can apply integration by parts, the expressions involved
behave essentially the same as in the K\"{a}hler case.

In this section, the proofs are given here, primarily along the lines of the
arguments in \cite{cchlw} for the conic Sasaki-Ricci flow.

Firstly, we observe a long-time solution to the conic Sasaki-Ricci flow from
the twisted Sasaki-Ricci flow (\ref{2023-1A}) as following:

\begin{proposition}
(\cite[Theorem 4.1 and Proposition 4.2]{lz}) There exists a sequence
$\{\varepsilon_{i}\}$ satisfying $\varepsilon_{i}\rightarrow0$ such that the
solution $\varphi_{\varepsilon_{i}}(t)$ of the twisted Sasaki-Ricci flow
(\ref{2023-1}) converges to the unique solution $\varphi(t)$ of the conic
Sasaki-Ricci flow (\ref{2023}) in $C_{B}^{\infty}$-topology on $[0,\infty
)\times(M\backslash D^{T})$. That is, $\omega\left(  t\right)  =\widehat
{\omega}_{0}+\sqrt{-1}\partial_{B}\overline{\partial}_{B}\varphi\left(
t\right)  $ is a long-time solution to the conic Sasaki-Ricci flow. Moreover,
for any $t\in\lbrack0,\infty),$ there exists a positive constant $\alpha<1$
such that the basic function $\varphi(t)$ is H\H{o}lder continuous with
respect to $\omega_{0}$ on $M$. And $||\varphi(t)||_{C_{B}^{\alpha}}$ is
bounded by a constant $C$ which depends only on $\beta,$ $\omega_{0}$,
$||\varphi(t)||_{C_{B}^{0}(M\backslash D^{T})}$ and $||\overset{\centerdot
}{\varphi}(t)||_{C_{B}^{0}(M\backslash D^{T})}$.
\end{proposition}

Next we will show that the transverse twisted Ricci potential $u_{\epsilon
}\left(  t\right)  $ which is a basic function behaves very well as
$t\rightarrow\infty$ under the twisted Sasaki-Ricci flow. This implies that
the limit of twisted Sasaki-Ricci flow should be the conic Sasaki-Ricci
soliton in the $L^{2}$-topology as in the next subsection. In fact, it follows
from the estimate of previous subsection and the twisted transverse Perelman's
$\mathcal{W}_{\theta_{\epsilon}}^{T}$-functional below that

\begin{proposition}
(\cite[Theorem 7 and Theorem 8]{cchlw}) Under the twisted Sasaki-Ricci flow

\begin{enumerate}
\item
\begin{equation}%
\begin{array}
[c]{c}%
\int_{M}|\nabla^{T}\nabla^{T}u_{\varepsilon}|^{2}(\omega^{\epsilon})^{n}%
\wedge\eta_{0}\rightarrow0
\end{array}
\label{a2}%
\end{equation}
as $t\rightarrow\infty.$

\item
\begin{equation}%
\begin{array}
[c]{c}%
\int_{M}(\Delta_{B}u_{\varepsilon}-|\nabla^{T}u_{\varepsilon}|^{2}%
+u_{\varepsilon}-a_{\varepsilon})^{2}(\omega^{\epsilon})^{n}\wedge\eta
_{0}\rightarrow0
\end{array}
\label{b2}%
\end{equation}
as $t\rightarrow\infty.$ Here $a_{\varepsilon}\left(  t\right)  =\frac{\beta
}{V}\int_{M}u_{\epsilon}\left(  t\right)  \exp\left(  -u_{\epsilon}\left(
t\right)  \right)  d\mu_{\epsilon}\left(  t\right)  $ as in (\ref{2023-6}).
\end{enumerate}
\end{proposition}

Next we define the twisted transverse Perelman's $\mathcal{W}_{\theta
_{\epsilon}}^{T}$-functional on a compact Sasakian $(2n+1)$-manifold by
\begin{equation}%
\begin{array}
[c]{l}%
\mathcal{W}_{\theta_{\epsilon}}^{T}(g^{T},f,\tau)\\
=(4\pi\tau)^{-n}\int_{M}(\tau(\mathrm{R}^{\mathrm{T}}(g(t))-tr_{g(t)}%
\theta_{\epsilon}+|\nabla^{T}f|^{2})+\beta f-2n)e^{-f}(\omega^{\epsilon}%
)^{n}\wedge\eta_{0},
\end{array}
\label{2022-2}%
\end{equation}
for $f\in C_{B}^{\infty}(M;\mathbb{R})$, $\tau>0.$ Also denote
\[%
\begin{array}
[c]{c}%
\lambda_{\varepsilon}^{T}(g^{T},\tau)=\inf\{\mathcal{W}_{\theta_{\epsilon}%
}^{T}(g^{T},f,\tau):f\in C_{B}^{\infty}(M;\mathbb{R});\text{ }\int_{M}%
(4\pi\tau)^{-n}e^{-f}(\omega^{\epsilon})^{n}\wedge\eta_{0}=V\}.
\end{array}
\]
Then we have
\[
-\infty<\lambda_{\varepsilon}^{T}(g^{T},\tau)\leq C
\]
and there exists $f_{\tau}\in C_{B}^{\infty}(M;\mathbb{R})$ so that
\[
\mathcal{W}_{\theta_{\epsilon}}^{T}(g^{T},f_{\tau},\tau)=\lambda_{\varepsilon
}^{T}(g^{T},\tau).
\]

Note that $\mathcal{W}_{\theta}^{T}$-functional can be expressed as
\[%
\begin{array}
[c]{ccl}%
\mathcal{W}_{\theta_{\epsilon}}^{T}(g^{T},f) & = & \mathcal{W}^{T}%
(g^{T},f,\frac{1}{2})+(2\pi)^{-n}(2n)V\\
& = & (2\pi)^{-n}\int_{M}(\frac{1}{2}(\mathrm{R}^{\mathrm{T}}-tr_{g(t)}%
\theta_{\epsilon}+|\nabla^{T}f|^{2})+\beta f)e^{-f}(\omega^{\epsilon}%
)^{n}\wedge\eta_{0},
\end{array}
\]
where $(g^{T},f)$ satisfies $\int_{M}e^{-f}(\omega^{\epsilon})^{n}\wedge
\eta_{0}=V$ and $\tau=\frac{1}{2}.$ Again
\[%
\begin{array}
[c]{c}%
\lambda_{\varepsilon}^{T}(g^{T})=\inf\{\mathcal{W}_{\theta_{\epsilon}}%
^{T}(g^{T},f):f_{\tau}\in C_{B}^{\infty}(M;\mathbb{R})\text{ \textrm{with}
}(2\pi)^{-n}\int_{M}e^{-f}(\omega^{\epsilon})^{n}\wedge\eta_{0}=V\}.
\end{array}
\]

\begin{proposition}
\label{P2024}(\cite{l}) Under the twisted Sasaki-Ricci flow, we have

\begin{enumerate}
\item
\begin{equation}%
\begin{array}
[c]{ccl}%
\frac{d}{dt}\lambda_{\varepsilon}^{T}(g^{T}) & = & \int_{M}|\mathrm{Ric}%
^{\mathrm{T}}-\theta_{\epsilon}+\overline{\nabla}^{T}\nabla^{T}f-\beta
g^{T}|^{2}e^{-f}\omega^{n}\wedge\eta_{0}\\
&  & +\frac{1}{2}\int_{M}\theta_{\epsilon}(\nabla^{T}f,\overline{\nabla}%
^{T}f)e^{-f}\omega^{n}\wedge\eta_{0}+\int_{M}|\nabla^{T}\nabla^{T}f|^{2}%
e^{-f}\omega^{n}\wedge\eta_{0}.
\end{array}
\label{52}%
\end{equation}
Here $f$ are minimizering solutions of $\lambda_{\varepsilon}^{T}(g^{T})$
associated to metrics $g^{T}$ and $\sigma^{T}$ is the family of transverse
diffeomorphisms of $M$ generated by the time-dependent vector field
$\nabla_{g^{T}}^{T}f.$

\item $g^{T}$ is a critical point of\ $\lambda_{\varepsilon}^{T}(g^{T})$ if
and only if $g^{T}$ is a gradient shrinking twisted Sasaki-Ricci soliton%
\begin{equation}%
\begin{array}
[c]{c}%
\mathrm{Ric}^{\mathrm{T}}-\theta_{\epsilon}+\nabla^{T}\overline{\nabla}%
^{T}f=\beta g^{T}\text{ \textrm{and} }\nabla^{T}\nabla^{T}f=0.
\end{array}
\label{53}%
\end{equation}
Here $f$ is a minimizer of $\mathcal{W}^{T}(g^{T},\cdot).$
\end{enumerate}
\end{proposition}

Now based on Perelman's non-collapsing theorem, we have

\begin{proposition}
\label{UP1} (\cite{lz}, \cite{co1}, \cite{chlw}) Under the twisted
Sasaki-Ricci flow, for any $t\geq1,$ $\varepsilon>0$ and
\[%
\begin{array}
[c]{c}%
-\mathrm{Ric}^{\mathrm{T}}(\omega_{\varepsilon}(t))+\beta\omega_{\varepsilon
}(t)+\theta_{\varepsilon}=\sqrt{-1}\partial_{B}\overline{\partial}%
_{B}u_{_{\varepsilon}}(t),
\end{array}
\]
we have

\begin{enumerate}
\item
\[
\mathrm{Vol}_{g_{\varepsilon}(t)}(B_{\xi,g_{\varepsilon}(t)}(x,r))\geq
Cr^{2n},
\]

\item
\[
|(\mathrm{R}^{T}(g_{\varepsilon}(t))-tr_{g_{\varepsilon}(t)}\theta
_{\varepsilon})|\leq C,
\]

\item
\[
||u_{\varepsilon}(t)||_{C^{1}(g_{\varepsilon}(t))}\leq C\ \ \mathrm{and}%
\ \ \mathrm{diam}(M,g_{\varepsilon}(t))\leq C.
\]

\end{enumerate}
\end{proposition}

As a consequence,

\begin{corollary}
\label{C31}(\cite{lz}, \cite{co2}) Under the twisted Sasaki-Ricci flow, we
have
\begin{equation}%
\begin{array}
[c]{l}%
\left(  \int_{M}f^{\frac{2(2n+1)}{2n-1}}d\mu_{\varepsilon}\right)
^{\frac{2n-1}{2n+1}}\leq A\int_{M}||\nabla^{T}f||_{g_{\varepsilon}(t)}^{2}%
d\mu_{\varepsilon}+B\int_{M}f^{2}d\mu_{\varepsilon}.
\end{array}
\label{1}%
\end{equation}

\end{corollary}

Combining these estimates of Proposition \ref{P41} and Corollary \ref{C31}, it
follows that

\begin{corollary}
There exist positive constants $A,$ $B,$ $C$ such that, for any $\varepsilon$
and $t,$

\begin{enumerate}
\item
\[%
\begin{array}
[c]{c}%
||\overset{\centerdot}{\phi}_{\epsilon}(t)||_{C_{B}^{0}}\leq C.
\end{array}
\]

\item
\begin{equation}%
\begin{array}
[c]{c}%
osc(\phi_{\varepsilon}(t))\leq\frac{A}{V}\int_{M}\phi_{\varepsilon}(t)d\mu
_{0}+B.
\end{array}
\label{4}%
\end{equation}

\end{enumerate}
\end{corollary}

By using (\ref{4}), (\ref{5}) and (\ref{6}) below, it follows that

\begin{proposition}
\label{P42}Suppose that the twisted transverse Mabuchi $K$-energy $K_{\eta
_{0},\theta_{\varepsilon}}$ are uniformly proper on
\[
\mathcal{H}(\omega_{0})=\{\phi\in C^{\infty}(M)|\omega_{0}+\sqrt{-1}%
\partial_{B}\overline{\partial}_{B}\phi>0\},
\]
that is
\begin{equation}
K_{\eta_{0},\theta_{\varepsilon}}(\phi)\geq f(J_{\eta_{0}}(\phi)), \label{3}%
\end{equation}
for any $\varepsilon$ and $\phi\in\mathcal{H}(\omega_{0}),$ where $f(t)$ is a
monotone increasing function with $\lim_{t\rightarrow\infty}f(t)=\infty.$
Then
\[
||\phi_{\varepsilon}(t)||_{C^{0}}\leq C
\]
for any $\varepsilon$ and $t.$
\end{proposition}

\subsection{Regularity of the Limit Space and Its Smooth Convergence}

Firstly, we will apply our previous results plus Cheeger-Colding-Tian
structure theory for K\"{a}hler manifolds (\cite{chlw}, \cite{cct}, \cite{tw2}
and \cite[Theorem2.3]{tz}) to study the structure of desired limit spaces.

Since $(M,\eta,\xi,\Phi,g,\omega)$ is a compact regular Sasakian manifold. By
the first structure theorem on Sasakian manifolds, $M$\ is a principal
$\mathbb{S}^{1}$-bundle over $Z$ which is also a $\mathbb{Q}$-factorial,
polarized, projective variety such that there is a Riemannian submersion$\ \pi
:(M,g,\omega)\rightarrow(Z,h,\omega_{h})$ with $g=g^{T}+\eta\otimes\eta,$
$g^{T}=\pi^{\ast}(h)$ and\ $\frac{1}{2}d\eta=\pi^{\ast}(\omega_{h}).$ The
orbit $\xi_{x}$ is compact for any $x\in M,$ we then define the transverse
distance function as
\[%
\begin{array}
[c]{c}%
d^{T}(x,y)\triangleq d_{g}(\xi_{x},\xi_{y})=d_{h}(\pi(x),\pi(y)),
\end{array}
\]
where $d$ is the distance function defined by the Sasaki metric $g.$ We define
a transverse ball $B_{\xi,g}(x,r)$ as follows:
\[%
\begin{array}
[c]{c}%
B_{\xi,g}(x,r)=\left\{  y:d^{T}(x,y)<r\right\}  =\left\{  y:d_{h}(\pi
(x),\pi(y))<r\right\}  .
\end{array}
\]
Note that when $r$ small enough, $B_{\xi,g}(x,r)$ is a trivial $\mathbb{S}%
^{1}$-bundle over the geodesic ball $B_{h}(\pi(x),r)$.

It follows from our previous work \cite{cchlw} and \cite{cht} (\cite{tw2} and
\cite[Theorem2.3]{tz}) that

\begin{theorem}
\label{T61} Let $(M_{i},(1-\beta)D^{T},\eta_{i},\xi,\Phi_{i},g_{i},\omega
_{i})$ be a sequence of log transverse Fano regular Sasakian $(2n+1)$%
-manifolds with conic Sasaki metrics $g_{i}=g_{i}^{T}+\eta_{i}\otimes\eta_{i}$
such that for basic potentials $\varphi_{i}$
\[%
\begin{array}
[c]{c}%
\eta_{i}=\eta+d_{B}^{C}\varphi_{i}%
\end{array}
\]
and
\[%
\begin{array}
[c]{c}%
d\eta_{i}=d\eta+\sqrt{-1}\partial_{B}\overline{\partial}_{B}\varphi_{i}.
\end{array}
\]
We denote that $(Z_{i},(1-\beta)D,h_{i},J_{i},\omega_{h_{i}})$ are a sequence
of projective K\"{a}hler manifolds which are the corresponding foliation leave
space with respect to $(M_{i},(1-\beta)D^{T},\eta_{i},\xi,\Phi_{i}%
,g_{i},\omega_{i})$ such that
\[%
\begin{array}
[c]{c}%
\omega_{g_{i}^{T}}=\frac{1}{2}d\eta_{i}=\pi^{\ast}(\omega_{h_{i}});\text{
}\Phi_{i}=\pi^{\ast}(J_{i}).
\end{array}
\]
Suppose that $(M_{i},(1-\beta)D^{T},\eta_{i},\xi,\Phi_{i},g_{i},\omega_{i})$
satisfies
\[%
\begin{array}
[c]{c}%
\int_{M}|\mathrm{Ric}_{g_{i}^{T}}^{\mathrm{T}}-(1-\beta)[D^{T}]|^{p}\omega
_{i}^{n}\wedge\eta\leq\Lambda,
\end{array}
\]
and
\[%
\begin{array}
[c]{c}%
\mathrm{Vol}(B_{\xi,g_{i}^{T}}(x_{i},1))\geq\upsilon
\end{array}
\]
for some $p>n,$ $\Lambda>0,$ $\upsilon>0$. Then passing to a subsequence if
necessary, $(M_{i},(1-\beta)D^{T},\Phi_{i},g_{i},x_{i})$ converges in the
Cheeger-Gromov sense to limit length spaces $(M_{\infty},(1-\beta)D_{\infty
}^{T},\Phi_{\infty},d_{\infty},x_{\infty})$ and the corresponding leave space
$(Z_{i},(1-\beta)D,J_{i},h_{i},\pi(x_{i}))$ converges to $(Z_{\infty}%
,(1-\beta)D_{\infty},J_{\infty},h_{\infty},\pi(x_{\infty}))$ such that

\begin{enumerate}
\item for any $r>0$ and $p_{i}\in M_{i}$ with $p_{i}\rightarrow p_{\infty}\in
M_{\infty},$%
\[%
\begin{array}
[c]{c}%
\mathrm{Vol}(B_{h_{i}}(\pi(p_{i}),r))\rightarrow\mathcal{H}^{2n}(B_{h_{\infty
}}(\pi(p_{\infty}),r))
\end{array}
\]
and
\[%
\begin{array}
[c]{c}%
\mathrm{Vol}(B_{\xi,g_{i}^{T}}(p_{i},r))\rightarrow\mathcal{H}^{2n}%
(B_{\xi,g_{\infty}^{T}}(p_{\infty},r)).
\end{array}
\]
Moreover,
\[%
\begin{array}
[c]{c}%
\mathrm{Vol}(B(p_{i},r))\rightarrow\mathcal{H}^{2n+1}(B(p_{\infty},r)),
\end{array}
\]
where $\mathcal{H}^{m}$ denotes the $m$-dimensional Hausdorff measure.

\item $M_{\infty}$ is a $\mathbb{S}^{1}$-orbibundle over the normal projective
variety $Z_{\infty}=M_{\infty}/\mathcal{F}_{\xi}.$

\item $Z_{\infty}=\mathcal{R}\cup\mathcal{S}$ is a normal projective variety
and $\mathcal{S=}\overline{\mathcal{S}}\cup D_{\infty}$ such that
$\overline{\mathcal{S}}$ is a closed singular set of at least codimension two
and $\mathcal{R}$ consists of points whose tangent cones are $\mathbb{R}%
^{2n}.$

\item The convergence on the regular part of $M_{\infty}$ which is a
$\mathbb{S}^{1}$-principle bundle over $\mathcal{R}$ in the $(C_{B}^{\alpha
}\cap L_{B}^{2,p})$-topology for any $0<\alpha<2-\frac{n}{p}$.
\end{enumerate}
\end{theorem}

Now by using the convergence theorem in Theorem \ref{T61}, we have the
regularity of the limit space. We define a family $g_{i}^{T}(t)$ of conic
Sasaki-Ricci flow by%
\[
(M,g_{i}^{T}(t))=(M,g^{T}(t_{i}+t))
\]
for $t\geq-1$ and $t_{i}\rightarrow\infty$ and for $g_{i}^{T}(t)=\pi^{\ast
}(h_{i}(t))$
\[
(Z,h_{i}(t))=(M,h_{i}(t_{i}+t)).
\]
Now for the associated transverse Ricci potential $u_{i}(t)$ as in Proposition
\ref{UP1}, we have
\[
||u_{i}(t)||_{C^{0}}+||\nabla^{T}u_{i}(t)||_{C^{0}}+||\Delta_{B}%
u_{i}(t)||_{C^{0}}\leq C.
\]
Moreover, it follows from (\ref{a2}) that
\[%
\begin{array}
[c]{c}%
\int_{M}|\nabla^{T}\nabla^{T}u_{i}|^{2}\omega(t)^{n}\wedge\eta_{0}\rightarrow0
\end{array}
\]
as $i\rightarrow\infty.$ Furthermore, by Theorem \ref{T31} and Theorem
\ref{T61}, passing to a subsequence if necessary, we have at $t=0,$%
\[
(M_{i},(1-\beta)D^{T},g_{i}^{T}(0))\rightarrow(M_{\infty},(1-\beta
)D^{T},g_{\infty}^{T},d_{\infty}^{T})
\]
such that
\[
(Z_{i},(1-\beta)D,h_{i}(0))\rightarrow(Z_{\infty},(1-\beta)D,h_{\infty
},d_{h_{\infty}})
\]
in the Cheeger-Gromov sense. Moreover,
\begin{equation}%
\begin{array}
[c]{c}%
(g_{i}^{T}(0),u_{i}(0))\overset{C_{B}^{\alpha}\cap L_{B}^{2,p}}{\rightarrow
}(g_{\infty}^{T},u_{\infty})
\end{array}
\label{c2}%
\end{equation}
on $(M_{\infty})_{reg}$ which is a $\mathbb{S}^{1}$-principle bundle over
$\mathcal{R}.$ The convergence of $u_{i}(0)$ on the regular part follows from
the elliptic regularity (\cite{co1}, \cite{co2}) of
\[%
\begin{array}
[c]{c}%
\Delta_{B}u_{i}(0)=n-\mathrm{R}^{\mathrm{T}}(g_{i}^{T}(0))\in L_{B}^{p}.
\end{array}
\]

\begin{theorem}
\label{T63} (\cite[Theorem 10]{cchlw}) Suppose (\ref{c2}) holds, then
$g_{\infty}^{T}$ is smooth and satisfies the conic Sasaki-Ricci soliton
equation%
\begin{equation}%
\begin{array}
[c]{c}%
\mathrm{Ric}^{\mathrm{T}}(g_{\infty}^{T})-(1-\beta)[D^{T}]+\mathrm{Hess}%
^{T}(u_{\infty})=\beta g_{\infty}^{T}%
\end{array}
\label{c1}%
\end{equation}
on $(M_{\infty})_{reg}$ which is a $\mathbb{S}^{1}$-bundle over $\mathcal{R}.$
Moreover, $\Phi_{\infty}$ is smooth and $g_{\infty}^{T}$ is K\"{a}hler with
respect to $\Phi_{\infty}=\pi^{\ast}(J_{\infty}).$
\end{theorem}

\begin{proof}
Since all $g_{\infty}^{T}$ and $u_{\infty}$ are basic, in the basic harmonic
foliation normal coordinates $\{x,x^{1},x^{2},\cdots,x^{2n}\}$ with
$z^{i}=x^{i}+\sqrt{-1}x^{n+i}$, the conic Sasaki-Ricci soliton equation
(\ref{c1}) is equivalent to
\begin{equation}%
\begin{array}
[c]{ll}
& (g^{T})^{\alpha\beta}\frac{\partial^{2}g_{\gamma\delta}^{T}}{\partial
x^{\beta}\partial x^{\alpha}}\\
= & \frac{\partial^{2}u_{\infty}}{\partial x^{\gamma}\partial x^{\delta}%
}+Q(g^{T},\partial g^{T})_{\gamma\delta}+T(g^{-1},\partial g^{T},\partial
u)_{\gamma\delta}-(1-\beta)\overset{0}{g}_{\gamma\delta}^{T}-\beta
g_{\gamma\delta}^{T}.
\end{array}
\label{c3}%
\end{equation}
By (\ref{a2}), (\ref{c3}) holds in $L_{B}^{2}((M_{\infty})_{reg})$. But
$g_{\infty}^{T}$ and $u_{\infty}$ are $L_{B}^{2,p}$, then (\ref{c3}) holds in
$L_{B}^{p}((M_{\infty})_{reg})$ too. On the other hand, by (\ref{b2}) and
(\ref{c1}), we have that
\begin{equation}%
\begin{array}
[c]{c}%
g_{\alpha\beta}^{T}\frac{\partial^{2}u_{\infty}}{\partial x^{\beta}\partial
x^{\alpha}}=(g^{T})^{\alpha\beta}\frac{\partial u_{\infty}}{\partial
x^{\alpha}}\frac{\partial u_{\infty}}{\partial x^{\alpha}}-2u_{\infty
}+2a_{\infty}%
\end{array}
\label{c4}%
\end{equation}
in $L_{B}^{p}((M_{\infty})_{reg}).$ Then a bootstrap argument as in \cite{pe}
to the elliptic systems (\ref{c3}) and (\ref{c4}) shows that $g_{\infty}^{T}$
and $u_{\infty}$ are smooth on $(M_{\infty})_{reg}$. The elliptic regularity
shows that $\Phi_{\infty}$ is smooth since $\nabla_{g_{\infty}^{T}}^{T}%
\Phi_{\infty}=0.$
\end{proof}

By applying the argument as in \cite[Theorem 1.2]{tz} (also \cite{cchlw}) to
the normal orbifold variety $Z_{\infty}$ which is mainly depended on the
Perelman's pseudolocality theorem (\cite{p1}), we have the smooth convergence
of the conic Sasaki-Ricci flow on the regular set $(M_{\infty})_{reg}$ which
is a $\mathbb{S}^{1}$-principle bundle over $\mathcal{R}$ and $\mathcal{R}$ is
the regular set of $Z_{\infty}$:

\begin{theorem}
\label{T64}The limit $(M_{\infty},(1-\beta)D^{T},d_{\infty})$ is smooth on the
regular set $(M_{\infty})_{reg}$ which is a $\mathbb{S}^{1}$-principle bundle
over the regular set $\mathcal{R}$ of $Z_{\infty}$ and $d_{\infty}^{T}$ is
induced by a smooth conic Sasaki-Ricci soliton $g_{\infty}^{T}$ and
$g^{T}(t_{i})$ converge to $g_{\infty}^{T}$ in the $C^{\infty}$-topology on
$(M_{\infty})_{reg}.$ Moreover, the singular set $\overline{\mathcal{S}}$ of
$Z_{\infty}$ is the codimension two orbifold singularities which are klt.
\end{theorem}

The proofs are straightforward from \cite{cchlw}. We give the proof here for completeness.

\begin{proof}
Note that since $g^{T}$ is a transverse K\"{a}hler metric and basic. It is
evolved by the conic Sasaki-Ricci flow, then it follows from the standard
computations as in \cite{co3} and \cite{lz} that
\[%
\begin{array}
[c]{c}%
\frac{\partial}{\partial t}\mathrm{Rm}^{\mathrm{T}}=\Delta_{B}\mathrm{Rm}%
^{\mathrm{T}}+\mathrm{Rm}^{\mathrm{T}}\ast\mathrm{Rm}^{\mathrm{T}}%
+\mathrm{Rm}^{\mathrm{T}}.
\end{array}
\]
Now by Perelman's pseudolocality theorem (\cite{p1}, \cite{tz}) with
Proposition \ref{UP1}, there exists $\varepsilon_{0},$ $\delta_{0},$ $r_{0}$
which depend on $p$ as in the Theorem \ref{T61} such that for any
$(x_{0},t_{0})$, if
\begin{equation}%
\begin{array}
[c]{c}%
\mathrm{Vol}(B_{\xi,g_{i}^{T}(t_{0})}(x_{0},r))\geq(1-\varepsilon
_{0})\mathrm{Vol}(B(0,r))
\end{array}
\label{e1}%
\end{equation}
for some $r\leq r_{0},$ where $\mathrm{Vol}(B(0,r))$ denotes the volume of
Euclidean ball of radius $r$ in $\mathbb{R}^{2n},$ then we have the following
curvature estimate
\begin{equation}%
\begin{array}
[c]{c}%
|\mathrm{Rm}^{\mathrm{T}}|_{g_{i}^{T}}(x,t)\leq\frac{1}{t-t_{0}}%
\end{array}
\label{e2}%
\end{equation}
for all $x\in B_{\xi,g_{i}^{T}(t)}(x_{0},\varepsilon_{0}r)$ and $t_{0}<t\leq
t_{0}+\varepsilon_{0}^{2}r^{2}$ and the volume estimate
\begin{equation}%
\begin{array}
[c]{c}%
\mathrm{Vol}(B_{\xi,g_{i}^{T}(t)}(x_{0},\delta_{0}\sqrt{t-t_{0}}))\geq
(1-\eta)\mathrm{Vol}(B(0,\delta_{0}\sqrt{t-t_{0}}))
\end{array}
\label{e3}%
\end{equation}
for $t_{0}<t\leq t_{0}+\varepsilon_{0}^{2}r^{2}$ and $\eta\leq\varepsilon_{0}$
is the constant such that the $C^{\alpha}$ harmonic radius at $x_{0}$ is
bounded below by $\delta_{0}\sqrt{t-t_{0}}.$

As in Theorem \ref{T61}, the metric $g_{i}^{T}(0)$ converges to $g_{\infty
}^{T}$ in the $(C^{\alpha}\cap L_{B}^{2,p})$-topology on $\mathcal{R}$. Now it
is our goal to show that the metric $g_{i}^{T}(0)$ converges smoothly to
$g_{\infty}^{T}$. For $0<r\leq r_{0}$ and $t\geq-1,$ define
\[%
\begin{array}
[c]{c}%
\Omega_{r,i,t}:=\{x\in M\text{ }|\text{ (\ref{e1}) \textrm{holds on} }%
B_{\xi,g_{i}^{T}(t)}(x,t)\}.
\end{array}
\]
Then (\ref{e3}) implies that
\[%
\begin{array}
[c]{c}%
\Omega_{r,i,t}\subset\Omega_{\delta_{0}\sqrt{s},i,t+s}%
\end{array}
\]
for $0<s\leq\varepsilon_{0}^{2}r^{2}.$

Let $r_{j}$ to be a decreasing sequence of radii such that $r_{j}\rightarrow0$
and $t_{j}=-\varepsilon_{0}r_{j}$. Then by applying \cite[(3.42)]{tz}, we may
assume that
\[%
\begin{array}
[c]{c}%
\Omega_{r_{j},i,t_{j}}\subset\Omega_{r_{j+1},i,t_{j+1}}.
\end{array}
\]
Then by (\ref{e2})%
\begin{equation}%
\begin{array}
[c]{c}%
||\mathrm{Rm}^{\mathrm{T}}||_{g_{i}^{T}(t)}(x,t)\leq\frac{1}{t-t_{j}}%
\end{array}
\label{e4}%
\end{equation}
for all $(x,t)$ with
\[%
\begin{array}
[c]{c}%
d_{g_{i}^{T}(t)}^{T}(x,\Omega_{r_{j},i,t_{j}})\leq\varepsilon_{0}r_{j},\text{
}t_{j}<t\leq0.
\end{array}
\]
By Shi's derivative estimate \cite{shi} to the curvature, there exists a
sequence of constants $C_{k,j,i}$ such that%
\begin{equation}%
\begin{array}
[c]{c}%
||(\nabla^{T})^{k}\mathrm{Rm}^{\mathrm{T}}||_{g_{i}^{T}(0)}(x,t)\leq C_{k,j,i}%
\end{array}
\label{e5}%
\end{equation}
on $\Omega_{r_{j},i,t_{j}}.$ Then passing to a subsequence if necessary, one
can find a subsequence $\{i_{j}\}$ of $\{j\}$ such that
\[%
\begin{array}
[c]{c}%
(\Omega_{r_{j},i_{j},t_{j}},g_{i_{j}}^{T}(t_{j}))\overset{C^{\alpha}%
}{\rightarrow}(\overline{\Omega},g_{\overline{\Omega}}^{T})
\end{array}
\]
and
\[%
\begin{array}
[c]{c}%
(\Omega_{r_{j},i_{j},t_{j}},g_{i_{j}}^{T}(0))\overset{C^{\infty}}{\rightarrow
}(\Omega,g_{\Omega}^{T}),
\end{array}
\]
where $(\overline{\Omega},g_{\overline{\Omega}}^{T})$ and $(\Omega,g_{\Omega
}^{T})$ are smooth Riemannian manifolds and
\[
(\Omega,g_{\Omega}^{T})\text{ \textrm{is isometric to} }((M_{\infty}%
)_{reg},d_{\infty}^{T}).
\]
On the other hand, as in Theorem \ref{T61}, we may also have
\[%
\begin{array}
[c]{c}%
(M,g_{i_{j}}^{T}(t_{j}))\overset{d_{G,H}^{T}.}{\rightarrow}(\overline
{M}_{\infty},\overline{d}_{\infty}^{T})
\end{array}
\]
with $\overline{Z}_{\infty}=\overline{\mathcal{R}}\cup\overline{\mathcal{S}}.$
Then%
\begin{equation}
(M_{\infty})_{reg}\text{ is the }\mathbb{S}^{1}\text{-principle bundle over
}\mathcal{R} \label{f1}%
\end{equation}
and
\begin{equation}
(\overline{M}_{\infty})_{reg}\text{ is the }\mathbb{S}^{1}\text{-principle
bundle over }\overline{\mathcal{R}}. \label{f2}%
\end{equation}
Moreover, by the continuity of volume under the Cheeger-Gromov convergence
(\cite[Claim 3.7]{tz}), we have
\begin{equation}
(\overline{\Omega},g_{\overline{\Omega}}^{T})\text{ is isometric to
}((\overline{M}_{\infty})_{reg},\overline{d}_{\infty}^{T}) \label{f3}%
\end{equation}
and then follows from (\ref{e4}) as in \cite[Claim 3.8]{tz} that
\begin{equation}
(\overline{\Omega},g_{\overline{\Omega}}^{T})\text{ is also isometric to
}(\Omega,g_{\Omega}^{T}). \label{f4}%
\end{equation}
Finally (\ref{f1}), (\ref{f2}), (\ref{f3}) and (\ref{f4}) imply the metric
$g_{i}^{T}(0)$ converges smoothly to $g_{\infty}^{T}$ on $\mathcal{R}$.
\end{proof}

Now we work on the partial $C^{0}$-estimate under the conical Sasaki-Ricci
flow. For the solution $(M,\xi_{0},\eta_{0},\omega(t),g^{T}(t))$ of the
conical Sasaki-Ricci flow and the line bundle $(L,h(t)),$
\[
L:=(K_{M}^{T})^{-1}-(1-\beta)[D^{T}]
\]
with the basic Hermitian metric $h(t)=\omega^{n}(t),$ we work on the evolution
of the basic transverse holomorphic line bundle $(L^{m},h^{m}(t))$ for a large
integer $m$ such that $L^{m}$ is very ample. We consider the basic embedding
(\cite{cchlw}) which is $\mathbb{S}^{1}$-equivariant with respect to the
weighted $\mathbb{C}^{\ast}$-action in $\mathbb{C}^{N_{m}+1}$
\[
\Psi:M\rightarrow(\mathbb{CP}^{N_{m}},\omega_{FS})
\]
defined by the orthonormal basic transverse holomorphic section $\{s_{0}%
,s_{1},\cdots,s_{N}\}$ in $H_{B}^{0}(M,L^{m})$ with $N_{m}=\dim H_{B}%
^{0}(M,L^{m})-1$ with
\[%
\begin{array}
[c]{c}%
\int_{M}(S_{i},S_{j})_{h_{t}^{m}}\omega_{t}^{n}\wedge\eta_{0}=\delta_{ij}.
\end{array}
\]

Define
\begin{equation}%
\begin{array}
[c]{c}%
\mathcal{F}_{m}(x,t):=\sum_{\alpha=0}^{N_{m}}||S(x)||_{h_{t}^{m}}^{2}.
\end{array}
\label{58}%
\end{equation}
We say that, under the conic Sasaki-Ricci flow, $\{(M(t),(1-\beta)D^{T}%
,\omega(t),g^{T}(t)),0\leq t<\infty\}$ has a partial $C^{0}$-estimate if
$L^{m}$ is very ample for a large integer $m$ and $\mathcal{F}_{m}(x,t)$ is a
uniformly bounded below function on $M\times\lbrack0,\infty).$

As in our previous estimate (\cite{cchlw}), based by the uniform bound of the
Sobolev constant (\ref{1}) for the basic function along the conic Sasaki-Ricci
flow and H\"{o}rmander's $L^{2}$-estimate to $\overline{\partial}_{B}%
$-operator on basic $(0,1)$-forms (\cite[Proposition 2.1]{chll}) on $H_{B}%
^{0}(M,L^{m})$, it follows straightforward from Theorem \ref{T61}, Theorem
\ref{T64}, (\cite{ds}), (\cite[(5.4)]{t5}), (\cite{d1}) and (\cite{d2}) that

\begin{theorem}
\label{T65}(\cite{cchlw})\ Suppose (\ref{c2}) holds, we have
\begin{equation}%
\begin{array}
[c]{c}%
\inf_{t_{i}}\inf_{x\in M}\mathcal{F}_{m}(x,t_{i})>0
\end{array}
\label{d11}%
\end{equation}
for a sequence of $m\rightarrow\infty.$
\end{theorem}

As a consequence of the first structure theorem for Sasakian manifolds and
Theorem \ref{T65}, the Gromov-Hausdorff limit $Z_{\infty}$ is a orbifold
variety embedded in some $\mathbb{CP}^{N}$ and the singular set $\overline
{\mathcal{S}}$ is a subvariety (\cite{ds}, \cite{tw1}). Then one can refine
the regularity of Theorem \ref{T64} as following:

\begin{corollary}
\label{C62} Let $(M,(1-\beta)D^{T},\xi,\eta_{0},g_{0})$ be a compact
transverse log Fano regular Sasakian manifold of dimension five and
$(Z_{0}=M/\mathcal{F}_{\xi},(1-\beta)D,h_{0},\omega_{h_{0}})$ denote the space
of leaves of the characteristic foliation which is a projective smooth
variety. Then under the conic Sasaki-Ricci flow, $(M_{\infty},(1-\beta)D^{T})$
is a $\mathbb{S}^{1}$-orbidbundle over $(Z_{\infty},(1-\beta)D)$ which is a
normal projective variety and the singular subvariety $\overline{\mathcal{S}}$
of $Z_{\infty}$ is a codimension two orbifold klt singularities.
\end{corollary}

Finally, Theorem \ref{T66} follows easily from Theorem \ref{T61}, Theorem
\ref{T63}, and Corollary \ref{C62}.

\section{The Transverse Log K-Polystable}

\subsection{Log Sasaki-Donaldson-Futaki Invariant}

All transverse quantities on Sasakian manifolds such as log
Sasaki-Donaldson-Futaki invariant, log Sasaki-Mabuchi $K$-energy can be viewed
as their K\"{a}hler counterparts restricted on basic forms and transverse
K\"{a}hler structure. That is, all the integrands are only involved with the
transverse K\"{a}hler structure. Hence, under the twisted or conic
Sasaki-Ricci flow, the Reeb vector field and the transverse holomorphic
structure are both invariant, and the metrics are bundle-like. Furthermore,
when one applies integration by parts, the expressions involved behave
essentially the same as in the K\"{a}hler case.

We will adapt the notions as in \cite{d2}, \cite{li} and \cite{lz} in our
Sasakian setting. Recall that, for the Hamiltonian holomorphic vector field
$V,$ $d\pi_{\alpha}(V)$ is a holomorphic vector field on $V_{\alpha}$ and the
complex valued Hamiltonian function $u_{_{V}}:=\sqrt{-1}\eta(V)$ satisfies
\[%
\begin{array}
[c]{c}%
\overline{\partial}_{B}u_{_{V}}=-\sqrt{-1}i_{_{V}}(\frac{1}{2}d\eta).
\end{array}
\]

\begin{definition}
The ordinary Sasaki-Futaki invariant is defined by
\[%
\begin{array}
[c]{c}%
f_{M}(V):=-n\int_{M}u_{_{V}}(\mathrm{Ric}_{\omega}^{\mathrm{T}}-\omega
)\omega^{n-1}\wedge\eta
\end{array}
\]
or
\[%
\begin{array}
[c]{c}%
f_{M}(V):=-\int_{M}u_{_{V}}(\mathrm{R}_{\omega}^{\mathrm{T}}-n)\omega
^{n}\wedge\eta.
\end{array}
\]

\end{definition}

Let $(M,\eta,\left(  1-\beta\right)  D^{T},\widehat{\omega})$ be a regular
Sasakian $(2n+1)$-manifold with the transverse cone angle $2\pi\beta$ along
the divisor $D^{T}\thicksim-K_{M}^{T}.$ Then the scalar curvature of
$\widehat{\omega}$ on $M\backslash D^{T}$ is defined by
\[%
\begin{array}
[c]{c}%
\mathrm{R}^{\mathrm{T}}(\widehat{\omega})=\widehat{g}^{i\overline{j}}%
\widehat{R}_{i\overline{j}}^{T}=n\int_{M}\mathrm{Ric}^{\mathrm{T}}%
(\widehat{\omega})\wedge\widehat{\omega}^{n-1}\wedge\eta/\int_{M}%
\widehat{\omega}^{n}\wedge\eta
\end{array}
\]
with
\[%
\begin{array}
[c]{c}%
\widehat{\omega}^{n}\wedge\eta=\frac{\omega_{0}^{n}\wedge\eta}{||S^{T}%
||^{2(1-\beta)}}%
\end{array}
\]
for some smooth Sasaki metric $\omega_{0}$ in $c_{1}^{B}(M)$. Here $S^{T}$ is
the defining section of $D^{T}.$ Define a norm $||s^{T}||$ for any basic
section $s^{T}\in\Gamma(M,L^{T},h)$ given by
\[%
\begin{array}
[c]{c}%
||s^{T}||_{h}=\sqrt{h_{\alpha}}|s_{\alpha}^{T}|
\end{array}
\]
for the orthonormal CR-holomorphic basic section $\{s_{\alpha}^{T}\}$ in
$H_{B}^{0}(M,L^{T})$ on $U_{\alpha}$ and
\[%
\begin{array}
[c]{c}%
c_{1}^{B}(L^{T},h)=-\sqrt{-1}\partial_{B}\overline{\partial}_{B}\log
h_{\alpha}\text{\ }\mathrm{on}\text{ }U_{\alpha}.
\end{array}
\]
Hence by the CR Poinc\'{a}re-Lelong formula%
\[%
\begin{array}
[c]{c}%
\sqrt{-1}\partial_{B}\overline{\partial}_{B}\log||s^{T}||_{h}^{2}=-c_{1}%
^{B}(L^{T},h)+[Z_{s^{T}}]
\end{array}
\]
for any CR-holomorphic line bundle $(L^{T},h).$ In particular%
\[%
\begin{array}
[c]{c}%
\sqrt{-1}\partial_{B}\overline{\partial}_{B}\log||s^{T}||^{2}=-c_{1}%
^{B}([D^{T}])+|D^{T}|.
\end{array}
\]
Hence
\begin{equation}%
\begin{array}
[c]{ccl}%
\mathrm{Ric}^{\mathrm{T}}(\widehat{\omega}) & = & \mathrm{Ric}^{\mathrm{T}%
}(\omega_{0})+\sqrt{-1}\partial_{B}\overline{\partial}_{B}\log||S^{T}%
||^{2(1-\beta)}\\
& = & \mathrm{Ric}^{\mathrm{T}}(\omega_{0})-(1-\beta)c_{1}^{B}([D^{T}%
])+(1-\beta)|D^{T}|
\end{array}
\label{2022-3}%
\end{equation}
and the scalar curvature of $\widehat{\omega}$ on $M\backslash D^{T}$
\[%
\begin{array}
[c]{c}%
\mathrm{R}^{\mathrm{T}}(\widehat{\omega})=n\int_{M}[\mathrm{Ric}^{\mathrm{T}%
}(\omega_{0})-(1-\beta)c_{1}^{B}([D^{T}])]\wedge\widehat{\omega}^{n-1}%
\wedge\eta/\int_{M}\widehat{\omega}^{n}\wedge\eta.
\end{array}
\]
Set
\[%
\begin{array}
[c]{c}%
\kappa:=\int_{M}[c_{1}^{B}(M)-(1-\beta)c_{1}^{B}([D^{T}])]\wedge c_{1}%
^{B}(M)^{n-1}\wedge\eta/\int_{M}c_{1}^{B}(M)^{n}\wedge\eta
\end{array}
\]
which is only depends on basic cohomological classes. We define
\[%
\begin{array}
[c]{c}%
\mathrm{Vol}(M):=\frac{1}{n!}\int_{M}c_{1}^{B}(M)^{n}\wedge\eta\text{
\ }\mathrm{and}\text{ \ }\mathrm{Vol}(D^{T}):=\frac{1}{(n-1)!}\int_{D^{T}%
}c_{1}^{B}(M)^{n-1}\wedge\eta.
\end{array}
\]
Moreover, since $\mathrm{Vol}(D^{T})=n\mathrm{Vol}(M)$, it follows from
\cite{li} that if $\mathrm{R}^{T}(\widehat{\omega})$ is constant $n\kappa$:
\[%
\begin{array}
[c]{ccl}%
n\kappa & = & n\int_{M}[c_{1}^{B}(M)-(1-\beta)c_{1}^{B}([D^{T}])]\wedge
c_{1}^{B}(M)^{n-1}\wedge\eta/\int_{M}c_{1}^{B}(M)^{n}\wedge\eta\\
& = & \{n\int_{M}c_{1}^{B}(M)^{n}\wedge\eta-(1-\beta)n\int_{M}[c_{1}%
^{B}([D^{T}])\wedge c_{1}^{B}(M)^{n-1}\wedge\eta]\}/\int_{M}c_{1}^{B}%
(M)^{n}\wedge\eta\\
& = & n-\frac{(1-\beta)\mathrm{Vol}(D^{T})}{\mathrm{Vol}(M)}%
\end{array}
\]
and then
\[
\kappa=\beta\text{.}%
\]

\begin{definition}
If $\widehat{\omega}$ is a conic Sasaki metric with the transverse conic angle
$2\pi\beta$ along the divisor $D^{T}\thicksim-K_{M}^{T}$, one can define log
Sasaki-Futaki invariant%
\[%
\begin{array}
[c]{c}%
f_{M,(1-\beta)D^{T}}(V):=-n\int_{M}\widehat{u}_{_{V}}(\mathrm{Ric}%
_{\widehat{\omega}}^{\mathrm{T}}-\beta\widehat{\omega})\wedge\widehat{\omega
}^{n-1}\wedge\eta,
\end{array}
\]
or
\[%
\begin{array}
[c]{c}%
f_{M,(1-\beta)D^{T}}(V):=-\int_{M}\widehat{u}_{_{V}}(\mathrm{R}_{\widehat
{\omega}}^{\mathrm{T}}-n\beta)\wedge\widehat{\omega}^{n}\wedge\eta,
\end{array}
\]
or%
\[%
\begin{array}
[c]{c}%
f_{M,(1-\beta)D^{T}}(V):=\int_{M}V(u_{\widehat{\omega}})\widehat{\omega}%
^{n}\wedge\eta.
\end{array}
\]

\end{definition}

One can also calculate the ordinary Futaki invariant using the conic metric
$\widehat{\omega}$. Indeed, for choosing the smooth transverse K\"{a}hler
metric $\omega$, then it follows from (\ref{2022-3}) that we can recapture the
log Sasaki-Futaki invariant defined by Donaldson:

\begin{definition}
(\cite{d2}, \cite{li}) The log Sasaki-Donaldson-Futaki invariant is defined by%
\[%
\begin{array}
[c]{l}%
f_{M,(1-\beta)D^{T}}(V):=f_{M}(V)+(1-\beta)[\int_{D^{T}}u_{_{V}}\omega
^{n-1}\wedge\eta-n\int_{M}u_{_{V}}\omega^{n}\wedge\eta].
\end{array}
\]

\end{definition}

\subsection{Log Sasaki-Mabuchi $K$-energy}

For $d\mu_{\phi}=(d\eta_{\phi})^{n}\wedge\eta_{0}$ and $\eta_{\phi}=\eta
_{0}+d_{B}^{C}\phi$ with $\phi_{t}$ is any path with $\phi_{0}=c$ and
$\phi_{1}=\phi,$ we recall that
\[%
\begin{array}
[c]{c}%
I_{\eta_{0}}(\phi):=\frac{1}{V}\int_{M}\phi(d\mu_{0}-d\mu_{\phi})
\end{array}
\]
and
\[%
\begin{array}
[c]{c}%
J_{\eta_{0}}(\phi):=\frac{1}{V}\int_{0}^{1}\int_{M}\overset{\centerdot}{\phi
}_{t}(d\mu_{0}-d\mu_{\phi_{t}})dt.
\end{array}
\]

\begin{definition}
Let $u_{\omega_{0}}$ be the transverse twisted Ricci potential of $\omega_{0}%
$, that is%
\[%
\begin{array}
[c]{c}%
-\mathrm{Ric}^{\mathrm{T}}(\omega_{0})+\beta\omega_{0}+\theta=\sqrt
{-1}\partial_{B}\overline{\partial}_{B}u_{\omega_{0}}.
\end{array}
\]

\begin{enumerate}
\item We recall that the twisted Sasaki-Mabuchi $K$-energy $K_{\omega
_{0},\theta}(\phi)$ on a compact transverse log Fano Sasakian $(2n+1)$%
-manifold $(M,D^{T})$ is defined by%
\[%
\begin{array}
[c]{ccl}%
K_{\omega_{0},\theta}(\phi) & := & -\beta(I_{\eta_{0}}(\phi)-J_{\eta_{0}}%
(\phi))-\frac{1}{V}\int_{M}u_{\omega_{0}}(d\mu_{0}-d\mu_{\phi})\\
&  & +\frac{1}{V}\int_{M}\log\frac{\omega_{_{\phi}}^{n}\wedge\eta_{0}}%
{\omega_{0}^{n}\wedge\eta_{0}}d\mu_{\phi}.
\end{array}
\]

\item We also denote that $K_{\omega_{0},\theta}(\phi)=K_{\omega_{0},\theta
}(\omega_{1}).$ As in (\ref{2023-5}), we also denote that
\[%
\begin{array}
[c]{c}%
K_{\omega_{1},\theta}(\omega_{2}):=K_{\omega_{1},(1-\beta)D^{T}}(\omega_{2}).
\end{array}
\]

\end{enumerate}
\end{definition}

It follows easily from the definition that

\begin{lemma}
\label{L21}(\cite{lz}, \cite{li}, \cite{cj})

\begin{enumerate}
\item $K_{\omega_{0},(1-\beta)D^{T}}$ is independent of the path $\phi_{t}.$
Furthermore, it satisfies the cocycle condition%
\[%
\begin{array}
[c]{c}%
K_{\omega_{1},(1-\beta)D^{T}}(\omega_{2})+K_{\omega_{2},(1-\beta)D^{T}}%
(\omega_{3})=K_{\omega_{1},(1-\beta)D^{T}}(\omega_{3}).
\end{array}
\]

\item We also can integrate the log Sasaki-Futaki invariant to get the log
Sasaki-Mabuchi $K$-energy. That is, for $\theta=(1-\beta)D^{T}$ and
$\omega_{t}=\omega+\sqrt{-1}\partial_{B}\overline{\partial}_{B}\phi_{t}$%
\begin{equation}%
\begin{array}
[c]{ccl}%
\frac{\partial}{\partial t}K_{\eta_{0},(1-\beta)D^{T}}(\phi_{t}) & = &
-\int_{M}\overset{\centerdot}{\phi}_{t}[\mathrm{R}^{\mathrm{T}}(\omega
_{t})-n]\omega_{t}^{n}\wedge\eta\\
&  & +(1-\beta)\int_{D^{T}}\overset{\centerdot}{\phi}_{t}\omega_{t}%
^{n-1}\wedge\eta\\
&  & -n(1-\beta)\int_{M}\overset{\centerdot}{\phi}_{t}\omega_{t}^{n}\wedge
\eta\\
& = & -\int_{M}\overset{\centerdot}{\phi}_{t}(\mathrm{R}_{\phi_{t}%
}^{\mathrm{T}}-\beta n-tr_{\phi_{t}}\theta)\omega_{t}^{n}\wedge\eta.
\end{array}
\label{2023-5}%
\end{equation}

\item
\[%
\begin{array}
[c]{c}%
\frac{\partial}{\partial t}I_{\eta_{0}}(\phi_{t})=\frac{1}{V}\int_{M}%
\overset{\centerdot}{\phi}_{t}(d\mu_{0}-d\mu_{\phi_{t}})-\frac{1}{V}\int
_{M}\phi_{t}\Delta_{B}\overset{\centerdot}{\phi}_{t}d\mu_{\phi_{t}}.
\end{array}
\]

\item
\[%
\begin{array}
[c]{c}%
\frac{\partial}{\partial t}J_{\eta_{0}}(\phi_{t})=\frac{1}{V}\int_{M}%
\overset{\centerdot}{\phi}_{t}(d\mu_{0}-d\mu_{\phi_{t}}).
\end{array}
\]

\end{enumerate}
\end{lemma}

Now by definition and Lemma \ref{L21} with the straightforward computation as
same as in \cite{cj}, we have

\begin{proposition}
\label{P41} Under the twisted Sasaki-Ricci flow (\ref{2023-1A}), for
$\phi_{\varepsilon}(t)=\varphi_{\varepsilon}(t)+k\chi_{\epsilon}(\epsilon
^{2}+\left\Vert S\right\Vert _{h}^{2})$

\begin{enumerate}
\item
\[%
\begin{array}
[c]{c}%
\frac{d}{dt}K_{\eta_{0},\theta_{\varepsilon}}(\phi_{\varepsilon}(t))=-\frac
{1}{V}\int_{M}|\nabla^{T}u_{\varepsilon}(t)|_{g_{\varepsilon}(t)}^{2}%
d\mu_{\phi_{\varepsilon}},
\end{array}
\]

\item
\begin{equation}%
\begin{array}
[c]{c}%
\frac{1}{nV}\int_{M}\phi_{\varepsilon}(t)d\mu_{\phi_{\varepsilon t}}-C\leq
J_{\eta_{0}}(\phi_{\varepsilon}(t))\leq\frac{1}{V}\int_{M}\phi_{\varepsilon
}(t)d\mu_{0}+C,
\end{array}
\label{5}%
\end{equation}

\item
\begin{equation}%
\begin{array}
[c]{c}%
\frac{1}{V}\int_{M}\phi_{\varepsilon}(t)d\mu_{0}\leq\frac{n}{V}\int_{M}%
(-\phi_{\varepsilon}(t))d\mu_{\phi_{\varepsilon t}}-(n+1)K_{\eta_{0}%
,\theta_{\varepsilon}}(\phi_{\varepsilon}(t))+C.
\end{array}
\label{6}%
\end{equation}

\end{enumerate}
\end{proposition}

\subsection{Transverse Log K-Polystable}

By the CR Kodaira embedding theorem for transverse log Fano quasi-regular
Sasakian manifolds, there exists an embedding
\[
\Psi:M\rightarrow(\mathbb{CP}^{N},\omega_{FS})
\]
defined by the basic transverse holomorphic section $\{s_{0},s_{1}%
,\cdots,s_{N}\}$ of $H_{B}^{0}(M,(L^{T})^{m})$ which is $\mathbb{S}^{1}%
$-equivariant with respect to the weighted $\mathbb{C}^{\ast}$-action in
$\mathbb{C}^{N+1}$ with $N=\dim H_{B}^{0}(M,(L^{T})^{m})-1$ for a large
positive integer $m.$ Here%
\[
L^{T}:=(K_{M}^{T})^{-1}-(1-\beta)D^{T}.
\]
Since $Z$ is also log Fano, there is an embedding
\[
\psi_{|mL_{\emph{Z}}|}:Z\rightarrow\mathcal{P(}H^{0}(Z,(L_{\emph{Z}})^{m})).
\]

Define
\[
\Psi_{|m(L^{T})|}=\psi_{|mL_{\emph{Z}}|}\circ\pi
\]
such that
\[
\Psi_{|m(L^{T})|}:M\rightarrow\mathcal{P(}H_{B}^{0}(M,(L^{T})^{m}).
\]
We define
\[
\mathrm{Diff}^{\mathrm{T}}(M)=\{\sigma\in\mathrm{Diff}(M)\text{\ }|\text{
}\sigma_{\ast}\xi=\xi\text{\ \textrm{and} }\sigma^{\ast}g^{T}=(\sigma^{\ast
}g)^{T}\}
\]
and
\[
SL^{T}(N+1,\mathbb{C})=SL(N+1,\mathbb{C})\cap\mathrm{Diff}^{\mathrm{T}}(M).
\]

Any other basis of $H_{B}^{0}(M,(L^{T})^{m})$ gives an embedding of the form
$\sigma^{T}\circ\Psi_{|m(L^{T})|}$ with $\sigma^{T}\in SL^{T}(N+1,\mathbb{C}%
).$ Now for any subgroup of the weighted $\mathbb{C}^{\ast}$-action
$G_{0}=\{\sigma^{T}(t)\}_{t\in\mathbb{C}^{\ast}}$ of $SL^{T}(N+1,\mathbb{C}),$
there is a unique limiting%
\[%
\begin{array}
[c]{c}%
M_{\infty}=\lim_{t\rightarrow0}\sigma^{T}(t)(M)\subset\mathbb{CP}^{N}%
\end{array}
\]
and
\[%
\begin{array}
[c]{c}%
D_{\infty}^{T}=\lim_{t\rightarrow0}\sigma^{T}(t)(D^{T}).
\end{array}
\]

If $M_{\infty}$ has its leave space $Z_{\infty}=M_{\infty}/\mathcal{F}_{\xi}$
which is a log Fano projective K\"{a}hler orbifold with at worst codimension
two orbifold singularities $\overline{\mathcal{S}}.$ then for an integer $l>0$
sufficiently large such that $(L_{\infty}^{T})^{l}$ is very-ample with the
$\mathbb{S}^{1}$-equivariant embedding with respective the weighted
$\mathbb{C}^{\ast}$-action in $\mathbb{C}^{N+1}$
\[
i:M_{\infty}\rightarrow(\mathbb{CP}^{N},\omega_{FS})
\]
with the Bergman metric $\overline{\omega}_{\infty}=\frac{1}{l}i^{\ast}%
(\omega_{FS}).$

Let $V$ be a Hamiltonian holomorphic vector field whose real part generates
the action by $\sigma^{T}(e^{-s})$. As the previous discussion $(M_{\infty
},(1-\beta)D_{\infty}^{T})$ is transverse log Fano, there is a log
Sasaki-Futaki invariant defined by%
\[%
\begin{array}
[c]{l}%
f_{M_{\infty},(1-\beta)D_{\infty}^{T}}(V):=f_{M_{\infty}}(V)+(1-\beta
)[\int_{D_{\infty}^{T}}u_{_{V}}\overline{\omega}_{\infty}^{n-1}\wedge\eta
_{0}-n\int_{M_{\infty}}u_{_{V}}\overline{\omega}_{\infty}^{n}\wedge\eta_{0}].
\end{array}
\]
Thus one can introduce the Sasaki analogue of the log $K$-polystable on
K\"{a}hler manifolds. In particular, if $M$ admits an Sasaki-Einstein, then
$\operatorname{Re}f_{M_{\infty},(1-\beta)D_{\infty}^{T}}(V)$ is nonnegative.

\begin{definition}
\label{d61}Let $(M,\xi,\eta,g,\omega,(1-\beta)D^{T})$ be a compact transverse
log Fano regular Sasakian manifold and its leave space $(Z,h,\omega
_{h},(1-\beta)D)$ be a log Fano smooth projective variety. We say that $M$ is
transverse log $K$-polystable with respect to $(L^{T})^{m}$ if the log
Sasaki-Donaldson-Futaki invariant
\[%
\begin{array}
[c]{c}%
\operatorname{Re}f_{M_{\infty},(1-\beta)D_{\infty}^{T}}(V)\geq
0\text{\ \ \textrm{or} \ }\operatorname{Re}f_{Z_{\infty},(1-\beta)D_{\infty}%
}(X)\geq0
\end{array}
\]
for any weighted $\mathbb{C}^{\ast}$-action $G_{0}=\{\sigma^{T}(t)\}_{t\in
\mathbb{C}^{\ast}}$ of $SL^{T}(N+1,\mathbb{C})$ with a log Fano normal variety
$Z_{\infty}=M_{\infty}/\mathcal{F}_{\xi}$ and the equality holds if and only
if $M_{\infty}$ is transverse biholomorphic to $M.$ We say that $M$ is
transverse log $K$-polystable if it is transverse log $K$-polystable for all
large positive integer $m$.
\end{definition}

With Definition \ref{d61} in mind, we are ready to show that the log
Sasaki-Marbuchi $K$-energy is bounded from below under the conic Sasaki-Ricci flow.

\begin{theorem}
\label{t2022} Let $(M,\xi,\eta_{0},g_{0},(1-\beta)D^{T})$ be a compact regular
log transverse Fano Sasakian manifold of dimension five and $(Z_{0}%
=M/\mathcal{F}_{\xi},h_{0},\omega_{h_{0}},(1-\beta)D)$ denote the space of
leaves of the characteristic foliation which is a log Fano projective
K\"{a}hler manifold. If $(M,(1-\beta)D^{T},L^{T})$ is transverse log
$K$-polystable, then the log Sasaki-Marbuchi $K$-energy is bounded from below
under the conic Sasaki-Ricci flow%
\begin{equation}%
\begin{array}
[c]{c}%
K_{\omega_{0},(1-\beta)D^{T}}(\omega_{t})\geq-C(g_{0}).
\end{array}
\label{54-b}%
\end{equation}

\end{theorem}

\begin{proof}
Firstly, it follows from Proposition \ref{P41} that $K_{\omega_{0}%
,(1-\beta)D^{T}}(\omega_{t})$ is non-increasing in $t$. So it suffices to show
a uniform lower bound of
\begin{equation}%
\begin{array}
[c]{c}%
K_{\omega_{0},(1-\beta)D^{T}}(\omega_{t_{i}})\geq-C.
\end{array}
\label{56}%
\end{equation}
Now if $(M,(1-\beta)D^{T},L^{T})$ is transverse log $K$-polystable, we fix an
integer $m>0$ sufficiently large such that $(L^{T})^{m}$ is very-ample. We can
define the $\mathbb{S}^{1}$-equivariant embedding $\Psi_{i}:M\rightarrow
(\mathbb{CP}^{N_{m}},\omega_{FS})$ with respective to the weighted
$\mathbb{C}^{\ast}$-action in $\mathbb{C}^{N_{m}+1}$ with $\overline{\omega
}_{i}=\frac{1}{m}\Psi_{i}^{\ast}(\omega_{FS})$ so that for any $i\geq1,$ there
exists a $G_{i}\in SL^{T}(N_{m}+1,\mathbb{C})$ such that $\Psi_{i}=G_{i}%
\circ\Psi_{1}.$ It follows from a result by S. Paul that%
\begin{equation}%
\begin{array}
[c]{c}%
K_{\overline{\omega}_{1},(1-\beta)D^{T}}(\overline{\omega}_{i})\geq-C.
\end{array}
\label{566}%
\end{equation}

By the cocycle condition of the log Sasaki-Mabuchi $K$-energy,%
\[%
\begin{array}
[c]{l}%
K_{\omega_{0},(1-\beta)D^{T}}(\omega_{i})+K_{\omega_{i},(1-\beta)D^{T}%
}(\overline{\omega}_{i})\\
=K_{\omega_{0},(1-\beta)D^{T}}(\overline{\omega}_{i})\\
=K_{\omega_{0},(1-\beta)D^{T}}(\overline{\omega}_{1})+K_{\overline{\omega}%
_{1},(1-\beta)D^{T}}(\overline{\omega}_{i})
\end{array}
\]
and then
\[%
\begin{array}
[c]{c}%
K_{\omega_{0},(1-\beta)D^{T}}(\omega_{i})+K_{\omega_{i},(1-\beta)D^{T}%
}(\overline{\omega}_{i})\geq-C.
\end{array}
\]
Hence, to show (\ref{56}), we only need to get an upper bound for
\begin{equation}%
\begin{array}
[c]{c}%
K_{\omega_{i},(1-\beta)D^{T}}(\overline{\omega}_{i})\leq C.
\end{array}
\label{57}%
\end{equation}
Let $H_{\omega_{0}}$ be the conic Ricci potential of $\omega_{0}$, that is,
$-\mathrm{Ric}^{\mathrm{T}}(\omega_{0})+\beta\omega_{0}+(1-\beta)D^{T}%
=\sqrt{-1}\partial_{B}\overline{\partial}_{B}H_{\omega_{0}}.$ We recall that
the log Sasaki-Mabuchi $K$-energy $K_{\omega_{0},(1-\beta)D^{T}}(\phi)$ on a
compact transverse log Fano Sasakian $(2n+1)$-manifold $(M,(1-\beta)D^{T})$ is
defined by%
\[%
\begin{array}
[c]{ccl}%
K_{\omega_{0},(1-\beta)D^{T}}(\phi) & := & -\beta(I_{\eta_{0}}(\phi
)-J_{\eta_{0}}(\phi))-\frac{1}{V}\int_{M}H_{\omega_{0}}(d\mu_{0}-d\mu_{\phi
})\\
&  & +\frac{1}{V}\int_{M}\log\frac{\omega_{_{\phi}}^{n}\wedge\eta_{0}}%
{\omega_{0}^{n}\wedge\eta_{0}}d\mu_{\phi}.
\end{array}
\]
For a fixed $m,$ we define
\[%
\begin{array}
[c]{c}%
\overline{\rho}_{i}(x):=\frac{1}{m}\rho_{t_{i},m}(x):=\frac{1}{m}%
\mathcal{F}_{m}(x,t_{i}).
\end{array}
\]
Here $\mathcal{F}_{m}(x,t_{i})$ as in (\ref{58}). Then
\[%
\begin{array}
[c]{c}%
\omega_{i}=\overline{\omega}_{i}+\sqrt{-1}\partial_{B}\overline{\partial}%
_{B}\overline{\rho}_{i}.
\end{array}
\]
It follows that the log Sasaki-Mabuchi $K$-energy has the following explicit
expression%
\begin{equation}%
\begin{array}
[c]{ll}
& K_{\omega_{i},(1-\beta)D^{T}}(\overline{\omega}_{i})\\
= & \int_{M}\log\frac{\overline{\omega}_{i}^{n}}{\omega_{i}^{n}}%
\overline{\omega}_{i}^{n}\wedge\eta_{0}+\int_{M}H_{\omega_{i}}(\overline
{\omega}_{i}^{n}-\omega_{i}^{n})\wedge\eta_{0}\\
& -\sqrt{-1}\sum_{k=0}^{n-1}\frac{n-k}{n+1}\int_{M}(\partial_{B}\rho_{i}%
\wedge\overline{\partial}_{B}\overline{\rho}_{i}\wedge\omega_{i}^{k}%
\wedge\overline{\omega}_{i}^{n-k-1})\wedge\eta_{0}\\
\leq & \int_{M}\log\frac{\overline{\omega}_{i}^{n}}{\omega_{i}^{n}}%
\overline{\omega}_{i}^{n}\wedge\eta_{0}+\int_{M}H_{\omega_{i}}(\overline
{\omega}_{i}^{n}-\omega_{i}^{n})\wedge\eta_{0}.\\
\leq & \int_{M}\log\frac{\overline{\omega}_{i}^{n}}{\omega_{i}^{n}}%
\overline{\omega}_{i}^{n}\wedge\eta_{0}+C.
\end{array}
\label{57a}%
\end{equation}
Here $H_{\overline{\omega}_{i}}$ is the transverse conic Sasaki-Ricci
potential under the conic Sasaki-Ricci flow and we have used the Perelman
estimate that $|H_{\omega_{i}}|\leq C.$ On the other hand, it follows from the
partial $C^{0}$-estimate and applying the gradient estimate that
$\overline{\omega}_{i}\leq C\omega_{i}.$Therefore (\ref{57a}) implies
(\ref{57}) and then we are done.
\end{proof}

Now in view of the K\"{a}hler counterparts restricted on basic functions and
transverse K\"{a}hler structure, we can have some a priori estimates for the
minimizing solution $f_{t}$ of (\ref{53}) under the conic Sasaki-Ricci\ flow.
We refer to the details estimates as in \cite{l}, \cite[Proposition
4.2]{tzhu2} and \cite{psswe} which improved the estimates as in Proposition
\ref{UP1} if the transverse Sasaki-Mabuchi $K$-energy is bounded from below.

\begin{lemma}
(\cite{l}) If in addition the log Sasaki-Marbuchi $K$-energy is bounded from
below on the space of transverse K\"{a}hler potentials, then we have

\begin{enumerate}
\item
\[%
\begin{array}
[c]{c}%
\lim_{t\rightarrow\infty}|u(t)|_{C^{0}(M\backslash(1-\beta)D^{T})}=0,
\end{array}
\]

\item
\[%
\begin{array}
[c]{c}%
\lim_{t\rightarrow\infty}||\nabla^{T}u(t)||_{C_{B}^{0}(M\backslash
(1-\beta)D^{T},g_{t}^{T})}=0,
\end{array}
\]

\item
\[%
\begin{array}
[c]{c}%
\lim_{t\rightarrow\infty}|\Delta_{B}u(t)|_{C^{0}(M\backslash(1-\beta)D^{T}%
)}=0.
\end{array}
\]

\end{enumerate}
\end{lemma}

\begin{corollary}
\label{c2022} If the log Sasaki-Marbuchi $K$-energy is bounded from below on
the space of transverse K\"{a}hler potentials, then there exists a sequence of
$t_{i}\in\lbrack i,i+1]$ such that

\begin{enumerate}
\item
\[%
\begin{array}
[c]{c}%
\lim_{t_{i}\rightarrow\infty}|\Delta_{B}f_{t_{i}}|_{L_{B}^{2}(M\backslash
(1-\beta)D^{T},g_{t_{i}}^{T})}=0,
\end{array}
\]

\item
\[%
\begin{array}
[c]{c}%
\lim_{t_{i}\rightarrow\infty}||\nabla^{T}f_{t_{i}}||_{L_{B}^{2}(M\backslash
(1-\beta)D^{T},g_{t_{i}}^{T})}=0,
\end{array}
\]

\item
\[%
\begin{array}
[c]{c}%
\lim_{t_{i}\rightarrow\infty}\int_{M}f_{t_{i}}e^{-f_{t_{i}}}\omega_{g_{t_{i}%
}^{T}}^{n}\wedge\eta_{0}=0.
\end{array}
\]
Moreover, we have
\begin{equation}%
\begin{array}
[c]{c}%
\lim_{t\rightarrow\infty}\lambda^{T}(g_{t}^{T})=(2\pi)^{-n}(2n)V=\sup
\{\lambda^{T}(\overline{g}^{T}):\overline{g}^{T}\in c_{1}^{B}(M,(1-\beta
)D^{T})\}.
\end{array}
\label{55}%
\end{equation}

\end{enumerate}
\end{corollary}

It follows from (\ref{53}), Proposition \ref{P42}, Corollary \ref{c2022} and
Theorem \ref{t2022} that $M_{\infty}$ must be conic Sasaki-Einstein. Then

\begin{theorem}
$(M,\xi,\eta_{0},g_{0},(1-\beta)D^{T})$ be a compact transverse log Fano
regular Sasakian manifold of dimension five and $(Z_{0}=M/\mathcal{F}_{\xi
},h_{0},\omega_{h_{0}},(1-\beta)D)$ be the space of leaves of the
characteristic foliation which is a log Fano projective K\"{a}hler manifold.
If $(M,(1-\beta)D^{T},L^{T})$ is transverse log $K$-polystable, then there
exists a conic Sasaki-Einstein metric on $(M,(1-\beta)D^{T})$.
\end{theorem}

\begin{proof}
The Lie algebra of all Hamiltonian holomorphic vector fields is reductive
(\cite{fow}, \cite{cds1}, \cite{cds2}, \cite{cds3}, \cite{t5}). If $M_{\infty
}$ is not equal to $M$, there is a family of the weighted $\mathbb{C}^{\ast}%
$-action $\{G(s)\}_{s\in\mathbb{C}^{\ast}}\subset SL^{T}(N_{m}+1,\mathbb{C})$
such that\ $\Psi_{s}=G(s)\circ\Psi_{1}$ converges to the $\mathbb{S}^{1}%
$-equivariant embedding of $M_{\infty}$ with respect to the weighted
$\mathbb{C}^{\ast}$-action in $(\mathbb{CP}^{N_{m}},\omega_{FS}).$ Then the
log Sasaki-Donaldson-Futaki invariant $f_{M_{\infty},(1-\beta)D_{\infty}^{T}%
}(V)$ of $M_{\infty}$ vanishes. On the other hand, by the assumption that $M$
is transverse log $K$-polystable, if $M_{\infty}$ is not equal to $M$, then
\[%
\begin{array}
[c]{c}%
\operatorname{Re}f_{M_{\infty},(1-\beta)D_{\infty}^{T}}>0.
\end{array}
\]
This is a contradiction. Hence $M_{\infty}=M$. Therefore there is a conic
Sasaki-Einstein metric on $(M,(1-\beta)D^{T}).$
\end{proof}


\begin{thebibliography}{99999}                                                                                            %


\bibitem[B]{b}D. Barden, \textit{Simply connected five-manifolds}, Ann. of
Math. (2) 82 (1965), 365-385.

\bibitem[BBEGZ]{bbegz}R. Berman, S. Boucksom, P. Eyssidieux, V. Guedj and A.
Zeriahi, \textit{K\"{a}hler-Einstein metrics and the K\"{a}hler-Ricci flow on
log Fano varieties}, J. Reine Angew. Math. 751 (2019), 27-89.

\bibitem[Ber]{ber}B. Berndtsson, \textit{A Brunn-Minkowski type inequality for
Fano manifolds and some uniqueness theorems in K\"{a}hler geometry}, Invent.
math. (2015) 200:149--200.

\bibitem[Bir]{bir}C. Birkar, \textit{Singularities of linear systems and
boundedness of Fano varieties}, Annals of Mathematics 193 (2021), 347-405.

\bibitem[BG]{bg}C. P. Boyer and K. Galicki, \textit{Sasaki geometry}, Oxford
Mathematical Monographs. Oxford University Press, Oxford (2008).

\bibitem[BGS]{bgs}C. P. Boyer , K. Galicki and S. Simanca, \textit{Canonical
Sasakian metrics,} Comm. Math. Phys. 279 (2008), no. 3, 705--733.

\bibitem[BM]{bm}S. Bando and T. Mabuchi, \textit{Uniqueness of Einstein
K\"{a}hler metrics modulo connected group actions}, in: Algebraic geometry
(Sendai 1985), Adv. Stud. Pure Math. 10, North-Holland, Amsterdam (1987), 11--40.

\bibitem[BvC]{bvc}C. Boyer and C. van Coevering,\textit{\ Relative }%
$K$\textit{-stability and extremal Sasaki metrics}, Math. Res. Lett.25 (2018),
no.1, 1--19.

\bibitem[Cao]{cao}H. Cao, \textit{Deformation of K\"{a}hler metrics to
K\"{a}hler-Einstein metrics on compact K\"{a}hler manifolds}, Invent. Math. 81
(1985), 359--372.

\bibitem[CC1]{cc1}J. Cheeger and T. H. Colding, \textit{Lower bounds on the
Ricci curvature and the almost rigidity of warped products}, Ann. Math., 144
(1996), 189-237.

\bibitem[CC2]{cc2}J. Cheeger and T. H. Colding, \textit{On the structure of
spaces with Ricci curvature bounded below I}, J. Diff. Geom., 46 (1997), 406-480.

\bibitem[CC3]{cc3}J. Cheeger and T. H. Colding, \textit{On the structure of
spaces with Ricci curvature bounded below II}, J. Diff. Geom., 54 (2000), 13-35.

\bibitem[CCHLW]{cchlw}D.-C. Chang, S.-C. Chang, Y. Han, C. Lin and C.-T. Wu,
\textit{Gradient shrinking Sasaki-Ricci solitons on Sasakian manifolds of
dimension up to seven}, arXiv: 2210.12702.

\bibitem[CCT]{cct}J. Cheeger, T. H. Colding and G. Tian,\textit{\ On the
singularities of spaces with bounded Ricci curvature}, Geom. Funct. Anal., 12
(2002), 873-914.

\bibitem[CDS1]{cds1}X. Chen, S. Donaldson, and S. Sun,
\textit{K\"{a}hler-Einstein metrics on Fano manifolds I}, J. Amer. Math. Soc.
28 (2015), no. 1, 183--197.

\bibitem[CDS2]{cds2}X. Chen, S. Donaldson, and S. Sun,
\textit{K\"{a}hler-Einstein metrics on Fano manifolds II}, J. Amer. Math. Soc.
28 (2015), no. 1, 199--234.

\bibitem[CDS3]{cds3}X. Chen, S. Donaldson, and S.
Sun,\textit{\ K\"{a}hler-Einstein metrics on Fano manifolds III}, J. Amer.
Math. Soc. 28 (2015), no. 1, 235--278.

\bibitem[CFO]{cfo}K. Cho, A. Futaki, and H. Ono, \textit{Uniqueness and
examples of compact toric Sasaki-Einstein metrics}, Commun. Math. Phys.
(2008), 277(2):439--458.

\bibitem[CHLL]{chll}S.-C. Chang, Y.-B. Han, N. Li and C. Lin,
\textit{Existence of nonconstant CR-holomorphic functions of polynomial growth
in Sasakian manifolds,} J. Reine Angew. Math. (Crelle Journal), vol. 2023, no.
802 (2023), DOI 10.1515/ crelle-2023-0046, pp. 223-254.

\bibitem[CHLW]{chlw}S.-C. Chang, Y. Han, C. Lin and C.-T. Wu,
\textit{Convergence of the Sasaki-Ricci flow on Sasakian }$5$%
\textit{-manifolds of general type}, arXiv:2203.00374.

\bibitem[CHT]{cht}S.-C. Chang, Y. Han and J. Tie, \textit{Compactness theorems
for Sasaki-Einsteins and Sasaki-Ricci solitons}, preprint.

\bibitem[CJ]{cj}T. Collins and A. Jacob, \textit{On the convergence of the
Sasaki-Ricci flow, Analysis}, complex geometry, and mathematical physics: in
honor of D. H. Phong, 11--21, Contemp. Math., 644, Amer. Math. Soc.,
Providence, RI, 2015.

\bibitem[CLW]{clw}S.-C. Chang, C. Lin and C.-T. Wu, \textit{Foliation
divisorial contraction by the Sasaki-Ricci flow on Sasakian }$5$%
\textit{-manifolds}, arXiv: 2203.01736.

\bibitem[CN]{cn}T. H. Colding and A. Naber, \textit{Sharp H\H{o}lder
continuity of tangent cones for spaces with a lower Ricci curvature bound and
applications}, Ann. of Math., 176 (2012), 1173-1229.

\bibitem[Co1]{co1}T. Collins, \textit{The transverse entropy functional and
the Sasaki-Ricci flow}, Trans. AMS., Volume 365, Number 3, March 2013, Pages 1277-1303.

\bibitem[Co2]{co2}T. Collins, \textit{Uniform Sobolev inequality along the
Sasaki-Ricci flow}, J. Geom. Anal. 24 (2014), 1323--1336.

\bibitem[Co3]{co3}T. Collins, \textit{Stability and convergence of the
Sasaki-Ricci flow}, J. reine angew. Math. 716 (2016), 1--27.

\bibitem[CS]{cs}J. Cable and Hendrik S\H{u}\ss , \textit{On the classification
of K\"{a}hler-Ricci solitons on Gorenstein del Pezzo surfaces}, European J. of
Math. (2018) 4:137--161.

\bibitem[CheSh]{chesh}I. A. Cheltsov and K. A. Shramov, \textit{Log-canonical
thresholds for smooth Fano threefolds}, with an appendix by J.-P. Demailly,
Uspekhi Mat. Nauk 63 (2008), no. 5 (383), 73--180.

\bibitem[CSW]{csw}X, Chen, S. Sun and B. Wang, \textit{K\"{a}hler-Ricci flow,
K\"{a}hler-Einstein metric, and K-stability}, Topol. 22 (2018) 3145-3173.

\bibitem[CSz1]{csz1}T. Collins and G. Szekelyhidi, \textit{K-semistability for
irregular Sasakian manifolds}, J. Differential Geometry 109 (2018) 81-109.

\bibitem[CSz2]{csz2}T. Collins and G. Szekelyhidi, \textit{Sasaki-Einstein
metrics and K-stability}, Geom. Topol. 23 (2019), no. 3, 1339--1413.

\bibitem[D2]{d1}S. K. Donaldson, \textit{Scalar curvature and stability of
toric varieties}, J. Differential Geom. 62 (2002), 289-349.

\bibitem[D2]{d2}S. K. Donaldson, \textit{K\"{a}hler metrics with cone
singularities along a divisor}, in Essays in Mathematics and Its Applications
(Springer, Heidelberg, 2012), pp. 49--79.

\bibitem[DP]{dp}J. P. Demailly and M. Paun, \textit{Numerical characterization
of the K\"{a}hler cone of a compact K\"{a}hler manifold}, Ann. of Math., 159
(2004), 1247-1274.

\bibitem[DK]{dk}J. P. Demailly and J. Kollar, \textit{Semi-continuity of
complex singularity exponents and K\"{a}hler-Einstein metrics on Fano
manifolds}, Ann. Ec. Norm. Sup 34 (2001), 525-556.

\bibitem[DS]{ds}S. Donaldson and S. Sun, \textit{Gromov-Hausdorff limits of
K\"{a}hler manifolds and algebraic geometry}, Acta Math. 213(1) (2014) 63--106.

\bibitem[DT]{dt}W. Ding and G. Tian, \textit{K\"{a}hler-Einstein metrics and
the generalized Futaki invariants}, Invent. Math., 110 (1992), 315-335.

\bibitem[EKA]{eka}A. El Kacimi-Alaoui, \textit{Operateurs transversalement
elliptiques sur un feuilletage riemannien et applications}, Compos. Math. 79
(1990) 57--106.

\bibitem[F]{f}A. Futaki, \textit{An obstruction to the existence of Einstein
K\"{a}hler metrics}, Invent. Math. 73 (1983), 437--443.

\bibitem[FOW]{fow}A. Futaki, H. Ono and G. Wang, \textit{Transverse K\"{a}hler
geometry of Sasaki manifolds and toric Sasaki-Einstein manifolds}, J. Diff.
Geom. 83 (2009) 585--635.

\bibitem[GMSW]{gmsw}J. P. Gauntlett, D. Martelli, J. Sparks and D. Waldram,
\textit{Sasaki-Einstein metrics on }$\mathbb{S}^{2}\times\mathbb{S}^{3}$, Adv.
Theor. Math. Phys. 8 (2004), 711--734.

\bibitem[GKN]{gkn}M. Godlinski, W. Kopczynski and P. Nurowski, \textit{Locally
Sasakian manifolds}, Classical Quantum Gravity 17 (2000) L105--L115.

\bibitem[GPSS]{gpss}B. Guo , D. H. Phong , J. Song\ and J. Sturm,
\textit{Compactness of K\"{a}hler-Ricci solitons on Fano manifolds}, Pure
Appl. Math. Q. 18 (2022), no. 1, 305--316.

\bibitem[H1]{h1}R. S. Hamilton, \textit{Three-manifolds with positive Ricci
curvature}, J. Diff. Geom. 17 (1982), no. 2, 255--306.

\bibitem[H2]{h2}R.S. Hamilton, \textit{The formation of singularities in the
Ricci flow}, in Surveys in differential geometry, Vol. II (Cambridge, MA,
1993), 7--136, Int. Press, Cambridge, MA, 1995.

\bibitem[K1]{k1}J. Kollar, \textit{Singular of pairs}, Algebraic geometry,
Santa Cruz 1995, Proc. Sympos. Pure Math., vol. 62, Amer. Math. Soc.,
Providence, RI, 1997, pp. 221-287.

\bibitem[K2]{k2}J. Kollar,\textit{\ Einstein metrics on connected sums of
}$\mathbb{S}^{2}\times\mathbb{S}^{3},$ J. Diff. Geom. 75 (2007), no. 2, 259--272.

\bibitem[K3]{k3}Kollar, \textit{Einstein metrics on five-dimensional Seifert
bundles}, J. Geom. Anal. 15 (2005), no. 3, 445--476.

\bibitem[L]{l}J. Liu, \textit{The generalized K\"{a}hler Ricci flow}, J. Math.
Anal. Appl. 408 (2013) 751--761.

\bibitem[Li]{li}C. Li, \textit{Remarks on logarithmic }$K$\textit{-stability},
Communications in Contemporary Mathematics, 17 (2014), no. 2, 1450020, 1-17.

\bibitem[LZ]{lz}J. Liu and X. Zhang, \textit{The conical K\"{a}hler-Ricci flow
on Fano manifolds}, Advances in Mathematics, 307(2017), 1324--1371.

\bibitem[M]{m}T. Mabuchi, \textit{K-energy maps integrating Futaki
invariants}, Tohoku Math. J., 38, 245-257 (1986).

\bibitem[MSY]{msy}D. Martelli, J. Sparks and S.-T. Yau,
\textit{Sasaki-Einstein manifolds and volume minimisation}, Commun. Math.
Phys. (2006), 268(1):39-65.

\bibitem[Na]{na}A. M. Nadel, \textit{Multiplier ideal sheaves and
K\"{a}hler-Einstein metrics of positive scalar curvature}, Ann. of Math. (2)
132 (1990), no. 3, 549-596.

\bibitem[P1]{p1}G. Perelman, \textit{The entropy formula for the Ricci flow
and its geometric applications}, preprint, arXiv: math.DG/0211159.

\bibitem[P2]{p2}G. Perelman, \textit{Ricci flow with surgery on
three-manifolds}, preprint, arXiv: math.DG/0303109.

\bibitem[P3]{p3}G. Perelman, \textit{Finite extinction time for the solutions
to the Ricci flow on certain three-manifolds}, preprint, arXiv: math.DG/0307245.

\bibitem[Paul1]{paul1}S. T. Paul,\textit{\ Hyperdiscriminant polytopes, Chow
polytopes, and Mabuchi energy asymptotics}, Ann. Math., 175 (2012), 255-296.

\bibitem[Paul2]{paul2}S. T. Paul, \textit{A numerical criterion for K-energy
maps of algebraic manifolds}, arXiv:1210.0924v1.

\bibitem[Pe]{pe}P. Petersen, \textit{Convergence theorems in Riemannian
geometry}, in \textquotedblright Comparison Geometry\textquotedblright\ edited
by K. Grove and P. Petersen, MSRI Publications, vol 30 (1997), Cambridge Univ.
Press, 167-202.

\bibitem[PSSWe]{psswe}D. H. Phong, J. Song, J. Sturm and B. Weinkove,
\textit{The K\"{a}hler-Ricci flow and }$\overline{\partial}$\textit{-operator
on vector fields}, J. Differ. Geom., 81 (2009), 631-647.

\bibitem[PW]{pw}J. Park and J. Won, \textit{Simply connected Sasaki-Einstein
rational homology }$5$\textit{-spheres}, Duke Math. J. 170(6) (2021), 1085-1112.

\bibitem[PW1]{pw1}P. Petersen and G.F. Wei,\textit{\ Relative volume
comparison with integral curvature bounds}, Geom. Funct. Anal., 7 (1997), 1031-1045.

\bibitem[PW2]{pw2}P. Petersen and G.F. Wei, \textit{Analysis and geometry on
manifolds with integral Ricci curvature bounds},\textit{\ II}, Trans. AMS.,
353 (2001), 457-478.

\bibitem[RT]{rt}J. Ross and R.P. Thomas, \textit{Weighted projective
embeddings, stability of orbifolds and constant scalar curvature K\"{a}hler
metrics,} JDG 88 (2011), no. 1, 109-160.

\bibitem[Ru]{ru}P. Rukimbira, \textit{Chern-Hamilton's conjecture and
K-contactness}, Houston J. Math. 21 (1995), no. 4, 709-718.

\bibitem[S]{s}S. Smale, \textit{On the structure of }$5$\textit{-manifolds,}
Ann. of Math. (2) 75 (1962), 38-46.

\bibitem[Shi]{shi}W.-X. Shi, \textit{Ricci deformation of the metric on
complete noncompact Riemannian manifolds}, J. Diff. Geom., 30 (1989), 303-394.

\bibitem[Sp]{sp}James Sparks, \textit{Sasaki-Einstein manifolds}, Surveys in
Differential Geometry 16 (2011), 265-324.

\bibitem[Su]{su}H. S\H{u}ss, \textit{On irregular Sasaki-Einstein metrics in
dimension }$5$, Sb. Math. 212 (2021), 1261; arXiv:1806.00285v1.

\bibitem[SWZ]{swz}K. Smoczyk, G. Wang and Y. Zhang, \textit{The Sasaki-Ricci
flow}, Internat. J. Math. 21 (2010), no. 7, 951--969.

\bibitem[SZ]{sz}Y. Shi and X. Zhu, \textit{K\"{a}hler-Ricci solitons on toric
Fano orbifolds}, Math. Z. (2012) 271:1241--1251.

\bibitem[T1]{t1}G. Tian, \textit{On Calabi's conjecture for complex surfaces
with positive first Chern class}, Invent. Math., 101, (1990), 101-172.

\bibitem[T2]{t2}G. Tian, \textit{Partial }$C^{0}$\textit{-estimates for
K\"{a}hler-Einstein metrics}, Commun. Math. Stat., 1 (2013), 105-113.

\bibitem[T3]{t3}G. Tian, \textit{K\"{a}hler-Einstein metrics with positive
scalar curvature}, Invent. Math. 137 (1997), no. 1, 1-37.

\bibitem[T4]{t4}G. Tian, \textit{Compactness Theorems for K\"{a}hler-Einstein
manifolds of dimension }$3$\textit{\ and up}, J. Diff. Geom., 35 (1992), 535-558.

\bibitem[T5]{t5}G. Tian, $K$\textit{-stability and K\"{a}hler-Einstein
metrics}, Comm. Pure Appl. Math. 68 (2015), no. 7, 1085--1156. Corrigendum:
Comm. Pure Appl. Math. 68 (2015), no. 11, 2082--2083.

\bibitem[TW1]{tw1}G. Tian and F. Wang, \textit{On the existence of conic
K\"{a}hler-Einstein metrics}, Adv. Math. 375 (2020), 107413, 42 pp.

\bibitem[TW2]{tw2}G. Tian and F. Wang, \textit{Cheeger-Colding-Tian theory for
conic K\"{a}hler-Einstein metrics}, J. Geom. Anal. 31 (2021), no. 2, 1471--1509.

\bibitem[TZ]{tz}G. Tian and Z. Zhang, \textit{Regularity of K\"{a}hler-Ricci
flow on Fano manifolds}, Acta Math. 216 (2016) No. 1, 127-176.

\bibitem[TZ2]{tz2}G. Tian and Z. Zhang, \textit{Degeneration of
K\"{a}hler-Ricci solitons}, Int. Math. Res. Not. IMRN 2012, no. 5, 957--985.

\bibitem[TZhu1]{tzhu1}G. Tian and X.-H. Zhu, \textit{Convergence of
K\"{a}hler-Ricci flow}, J. Amer. Math. Soc. 20 (2007), 675-699.

\bibitem[TZhu2]{tzhu2}G. Tian and X.-H. Zhu, \textit{Convergence of
K\"{a}hler-Ricci flow on Fano manifolds II}, J. Reine Angew. Math.. 678
(2013), 223-245.

\bibitem[WZ]{wz}X. Wang and X. Zhu, \textit{K\"{a}hler-Ricci solitons on toric
Fano manifolds with positive first Chern class}, Adv. Math. 188 (2004), 87--103.

\bibitem[Y1]{y1}S.-T. Yau, \textit{On the Ricci curvature of a compact
K\"{a}hler manifold and the complex Monge-Ampere equation, I}, Comm. Pure.
Appl. Math. 31 (1978), 339-411.

\bibitem[Y2]{y2}S.-T. Yau, \textit{Open problems in geometry}, Proc. Symp.
Pure Math. 54 (1993), 1-18.
\end{thebibliography}
\end{document}